\documentclass{article}
\usepackage[utf8]{inputenc}
\usepackage{amsmath}
\usepackage{amssymb}
\usepackage{mathtools}
\usepackage{tikz-cd}

\usepackage{quiver}

\usepackage{tikz}
\usepackage[margin=3.5cm]{geometry}
\usepackage[T1]{fontenc}
\usepackage{palatino}
\usepackage{graphicx}
\usepackage{amsthm}
\usepackage{emptypage}
\usepackage{lipsum}
\usepackage{enumerate}

\usepackage{pdfpages}
\usepackage{titlepic}

\theoremstyle{plain}
\newtheorem{theorem}[equation]{Theorem}
\newtheorem*{theorem*}{Theorem}

\newtheorem*{conjecture*}{Conjecture}

\newtheorem{lemma}[equation]{Lemma}
\newtheorem{proposition}[equation]{Proposition}
\newtheorem*{proposition*}{Proposition}
\newtheorem{corollary}[equation]{Corollary}

\theoremstyle{definition}
\newtheorem{definition}[equation]{Definition}
\newtheorem*{definition*}{Definition}
\newtheorem{construction}[equation]{Construction}

\theoremstyle{remark}
\newtheorem{remark}[equation]{Remark}
\newtheorem{example}[equation]{Example}

\usepackage{enumitem}
\usepackage{hyperref}

\usepackage{chngcntr}
\counterwithin{equation}{section}
\swapnumbers

\def\mainmatter{%
    \pagenumbering{arabic}
    \setcounter{page}{1}
    \setcounter{section}{0}
    \renewcommand{\thesection}{\arabic{section}}
}%

\usepackage[
backend=biber,
style=alphabetic
]{biblatex}
\DeclareNameAlias{default}{last-first}

\newcommand{\id}{\mathrm{id}}

\newcommand{\Hom}{\mathrm{Hom}}
\newcommand{\Aut}{\mathrm{Aut}}

\newcommand{\catname}[1]{{\normalfont\textbf{#1}}}

\newcommand{\Vecbun}{\catname{vect}}

\newcommand{\Octalg}{\catname{octalg}}
\newcommand{\Octalgobj}{\catname{octobj}}

\newcommand{\proj}{\catname{proj}}

\newcommand{\GL}{\mathrm{GL}}

\newcommand{\SO}{\mathrm{SO}}
\newcommand{\Orth}{\mathrm{O}}

\newcommand{\BC}{{\mathbb{C}}}

\newcommand{\BO}{{\mathbb{O}}}
\newcommand{\BP}{{\mathbb{P}}}

\newcommand{\BR}{{\mathbb{R}}}
\newcommand{\BS}{{\mathbb{S}}}

\newcommand{\BZ}{{\mathbb{Z}}}

\newcommand{\CC}{{\mathcal C}}

\DeclareFontFamily{OT1}{rsfs}{}
\DeclareFontShape{OT1}{rsfs}{n}{it}{<-> rsfs10}{}
\DeclareMathAlphabet{\curly}{OT1}{rsfs}{n}{it}

\newcommand\Id{\operatorname{Id}}

\title{$G_2$-structures as Octonion Algebras}
\author{Isak Sundelius}
\date{April 2026}

\begin{document}

\maketitle

\mainmatter

\begin{abstract}
    We define the category of $G_2$-structures over a Riemannian 7-manifold $M$ and present an isomorphism between this category and a full subcategory of the category of octonion algebras over the ring of smooth real-valued functions $C^\infty(M)$ of the same manifold $M$. A classification of $G_2$-structures in the same metric class is shown to agree with a parametrisation of octonion algebras with isometric norm. A short study of the local structure of octonion algebras over $C^\infty(M)$ shows similarities to the theory of octonion algebras over $\BR$. Thus, many of the results on real octonion algebras, and in general octonion algebras over rings, can be applied to $G_2$-structures viewed as octonion algebras, under the aforementioned isomorphism of categories.
\end{abstract}

\tableofcontents

\section*{Introduction}

The study of $G$-structures is a rich area of geometry which yields a clear division of different subfields of geometry, from the perspective of differential geometry. For a smooth manifold $M$ with tangent bundle $TM$, it is natural to study the frame bundle $F(TM)$, the bundle of local basis choices for the tangent bundle, and its reductions. In particular, the frame bundle is a $\GL_n(\BR)$-principal bundle, with $n=\dim(M)$. This means that for every Lie subgroup $G\subseteq \GL_n(\BR)$ we can inspect whether such a manifold $M$ admits a reduction to (i.e., is everywhere patched together by elements of) the subgroup $G$. This provides deep insight into the geometry of the underlying manifold in the package of a Lie group and, for a finer study of the geometry, one can define the intrinsic torsion of a $G$-structure.

An important example of this is is when an $n$-dimensional smooth manifold $M$ admits an $O_n(\BR)$-reduction, in which case it admits a Riemannian metric, hence the structure of a Riemannian manifold. Thus, Riemannian geometry is the study of manifolds with $O_n(\BR)$-structure. Another important example is $\GL_n(\BC)$-structures on a $2n$-dimensional smooth manifold $M$, which are called almost complex structures, and which has intrinsic torsion corresponding to the Nijenhuis tensor. When this tensor vanishes, the given $\GL_n(\BC)$-structure is a complex structure, in which case the manifold is called a complex manifold. The same goes for symplectic manifolds, Kähler manifolds and so on.

The group $G_2$ is one of the exceptional simple Lie groups over $\BR$, and has a compact and non-compact version, denoted $G_2$ and $G_2^*$ respectively. Both are Lie subgroups of $\GL_7(\BR)$ and thus one may study $G_2$-structures, which can then only be admitted by manifolds of dimension 7. Manifolds that admit a $G_2$-structure have been studied thoroughly, for an overview see \cite{Bry}. They are of particular importance in string and M-theory, \cite{PhysRevD.104.126014}, \cite{acharya_m_2004}. 

The group $G_2$ also appears, as an algebraic group, in the study of composition algebras. A composition algebra over a field $k$ is, roughly, a unital algebra over $k$ that is not necessarily commutative or associative and that is endowed with a quadratic form that commutes with the multiplication. Two composition algebras are called isomorphic if there is an isomorphism of their underlying vector spaces that preserves their multiplications, identity elements and quadratic forms. In particular, the group $G_2$ appears as the automorphism group of a composition algebra of rank 8 over a field, a so-called octonion algebra. A source for this theory is \cite{SV}. A foundational result in the study of composition algebras over fields is that they only exist in dimension 1, 2, 4 and 8, a fact sometimes called the Hurwitz theorem. The latter three have a variant that is division and one which is not (in which case it is called split). Over the real numbers the composition algebras which are division are the real numbers, the complex numbers, the quaternions and the octonions, also called the Cayley numbers, respectively. In fact, it is the existence of composition algebras that guarantees the existence of cross-products in dimension one less than the associated composition algebra, i.e., 0, 1, 3 and 7 (and these are the corresponding products restricted to the imaginary part). The theory of composition algebras over fields is classical and goes back to Cayley. The connection between $G_2$-structures and octonion algebras goes back to Fernández and Gray \cite{fernandez_riemannian_1982} in 1982, who use the Cayley product and the associated cross product to construct torsion classes of $G_2$-structures. The fact that the classical theory of octonion algebras is well-understood, and has a history of application to $G_2$-structures, is a key motivation for the perspective of this article. One of the goals is to create a situation in which much of the modern theory, as well as the classical theory, of octonion algebras can be applied. This situation is reached and described in terms of modern tools.

Composition algebras, and in effect the octonions, have been defined in a more general context, namely as algebras over an arbitrary ring instead of a field. In this case, the multiplication of an octonion algebra is no longer necessarily determined by its quadratic form, which is the case over fields and, moreover, local rings, see \cite{bix}. Still, the (algebraic) group $G_2$ plays a major role in the classification of octonion algebras over a given ring with a given quadratic form.

Sergei Grigorian has continued the algebraic study of $G_2$-structures by constructing certain vector bundles from $G_2$-structures. The main source for his algebraic treatment is \cite{Gri}. Sections of these so-called octonion bundles are used to provide an algebraic expression for the different $G_2$-structures which have the same associated metric.

This article aims to combine this with the study of octonion algebras over rings by providing an alternative definition of octonion bundles, characterising them as certain octonion algebras over $C^\infty(M)$. This is done by identifying vector bundles over $M$ with projective modules over $C^\infty(M)$, and studying the structure on fibres of such bundles we get that much of the theory of octonion algebras over rings, and, locally, the classical theory of octonion algebras over $\BR$, can be applied. These definitions and results are made by use of category theory, so we provide definitions of various categories. This also makes it possible to upgrade previously known 1:1-correspondences to equivalences and isomorphisms of categories, most importantly between $G_2$-structures and octonion bundles. Among other facts, this gives us an expression for $G_2$-metrics on fibres.

We first define octonion algebra objects in the category of vector bundles as a vector bundle with multiplication, identity and bilinear form. An octonion bundle over a Riemannian manifold of dimension 7 is then an octonion algebra object that satisfies some additional compatibility conditions with the underlying manifold. These are definitions \ref{octalgobjdef} and \ref{octbundef} respectively. Due to there being an equivalence of categories $\Vecbun(M)\simeq\proj_{C^\infty(M)}$ between the category of vector bundles and that of finitely generated projective modules over $C^\infty(M)$, we get that the category of octonion algebra objects in $\Vecbun(M)$ is equivalent to the category of octonion algebras over the ring $C^\infty(M)$. The category of octonion bundles can then be considered a full subcategory of the category of octonion algebras over $C^\infty(M)$.

Despite the fact that the ring $C^\infty(M)$ lacks many nice algebraic properties, the equivalence between octonion algebra objects and octonion algebras gives us two perspectives on localisation. On the side of octonion algebra objects, we get that the data over a point on $M$ is a real octonion algebra, in the classical sense and as presented by Springer and Veldkamp in \cite{SV}. This can also be seen as the localisation of the octonion algebra to the residue field over that point. Thus, many results from the classical theory can be lifted to octonion algebra objects and, in effect, octonion bundles. Some such results for this article are the relations appearing in proposition \ref{relationslemma}.

The main result of this article is that the category of $G_2$-structures (as defined in this paper) is equivalent to a full subcategory of the category of octonion algebras over the ring $C^\infty(M)$:
\begin{proposition*}[Proposition \ref{equivalencetheorem}]
    There is an isomorphism of categories
    \begin{align*}
        \catname{G2str}(M)\overset{\sim}{\longrightarrow} \catname{octbun}_{C^\infty(M)}.
    \end{align*}
\end{proposition*}

The article is structured as follows: In section \ref{prelim} we present the preliminaries for the article, recalling important results. In particular, subsection \ref{VBasmod} presents the results from the theory of smooth vector bundles, mainly to characterise them as projective modules over the ring $C^\infty(M)$ of real-valued smooth functions on $M$. The main source for this subsection is \cite{Nes}.

In the two subsequent sections $G_2$-structures and octonion algebras are presented. The main source for the definition of and results on $G_2$-structures is \cite{Bry}. It is most importantly used to characterise $G_2$-structures both as reductions of the structure group and as 3-forms satisfying certain conditions in terms of $G_2$. A definition of category of $G_2$-structures is also presented, in particular in a way that preserves the main two characterisations of $G_2$-structures.

Both the classical theory of octonion algebras over fields as well as the more modern theory of octonion algebras over rings are used later in the article. The main source for the former is \cite{SV} and the main sources for the latter are \cite{LPR}, \cite{garibaldi_albert_2024} and \cite{AlGi}. In particular, the source for the parametrisation of octonion algebras over rings, with isometric quadratic forms, is \cite{AlGi}. Similar to the preceding subsection, a notion of category of octonion algebras (over a ring) is presented.

In section \ref{obsect} the equivalence between $G_2$-structures and certain octonion algebras is made precise. This leads to the definition of octonion algebra objects over a manifold $M$. Subsection \ref{octbunintrinsic} provides a description of octonion bundles without reference to a $G_2$-structure.

In section \ref{twistsection} we use the results of section \ref{obsect} to produce a proof of the equality of the possible twists of $G_2$-structures. This proof explicitly uses octonion algebras over $c^\infty(M)$. It also relies on the results for octonion bundles that come from octonion algebras over $\BR$, mainly \ref{relationslemma}.

In section \ref{generalisationsection} we generalise the results of the above sections such that any octonion algebra over $C^\infty(M)$ for $M$ a smooth manifold can be realised as some kind of $G_2$- or related structure.

\subsection*{Notation and conventions}\label{conventions}

\begin{itemize}
    \item $M$: A connected smooth manifold.
    \item $C^\infty(M)$: The ring of smooth $\BR$-valued functions on a connected smooth manifold $M$.
    \item $\Vecbun(M)$: The category of smooth vector bundles of finite rank over $M$.
    \item $\Gamma(F\overset{\pi}{\to}M)$: The global sections of a vector bundle $F\overset{\pi}{\to}M$.
    \item $\proj_R$: The category of finitely generated projective modules over a commutative ring $R$.
    \item $\Octalg_R$: The category of octonion algebras over $R$, \ref{octalgcatdef}.
    \item $\catname{G2str}(M)$: The category of $G_2$-structures over a Riemannian manifold $M$, \ref{g2catprinc}.
    \item $\catname{def3form}$: The category of definite 3-forms over a Riemannian manifold $M$, \ref{g2cat3form}.
    \item $\catname{octobj}_{\Vecbun(M)}$: The category of octonion algebra objects in the category $\Vecbun(M)$ of vector bundles over a smooth manifold $M$, \ref{octalgobjcatdef}, \ref{octalgobjdef}.
\end{itemize}

We will sometimes not distinguish between quadratic forms and their associated bilinear forms, when the translation is clear from the context. An algebra is assumed to be unital but is not assumed to be commutative or associative. Rings are assumed to be unital, associative and commutative.

\section{Preliminaries on Vector Bundles}\label{prelim}

\subsection{Vector Bundles as Modules}\label{VBasmod}

We recall some foundational results about smooth vector bundles over smooth manifolds. We will let $M$ denote a smooth connected manifold. The main result is an equivalence of categories $\Vecbun(M)\overset{\sim}{\to}\proj_{C^\infty(M)}$, making it possible to characterise geometric structures as algebraic ones. The main source for this section is \cite{Nes}.

We begin by recalling that a global section of a smooth vector bundle $F\overset{\pi}{\to} M$ over $M$ is a smooth map $s:M\to F$ such that $\pi\circ s=\id_M$. Denote by $\Gamma(F)$ the set of global sections of $F$ over $M$. This set has the natural structure of a $C^\infty(M)$-module. For instance, if $\Lambda^0$ denotes the trivial bundle $M\times \BR$ over $M$, then $\Gamma(\Lambda^0)=C^\infty(M)$. Furthermore, $\Gamma$ sends morphisms of vector bundles to morphisms of modules, so it determines a functor. The first two of the following theorems are stated without proof.

\begin{theorem}\cite[Theorem 12.29]{Nes}\label{nesth1}
    For any pair $F$ and $G$ of vector bundles over $M$, the functor $\Gamma$ determines a 1:1-correspondence
    \begin{align*}
        \Hom_{\Vecbun(M)}(F, G)\cong\Hom_{C^\infty(M)}(\Gamma(F), \Gamma(G)).
    \end{align*}
\end{theorem}

\begin{theorem}\cite[Theorem 12.32]{Nes}\label{nesth2}
    Suppose $P$ is a $C^\infty(M)$-module. Then $P$ is isomorphic to the module of sections $\Gamma(F)$ of a smooth vector bundle $F$ over $M$ if and only $P$ is finitely generated and projective.
\end{theorem}

\begin{theorem}\cite[12.33. Equivalence of the two categories]{Nes}
     There is an equivalence of the category $\Vecbun(M)$ of vector bundles over the manifold $M$ with the category $\proj_{C^\infty(M)}$ of finitely generated projective $C^\infty(M)$-modules.
\end{theorem}

\begin{proof}
    Theorem \ref{nesth1} and theorem \ref{nesth2}, together with the fact that every module of sections $\Gamma(F)$ of a vector bundle $F$ over $M$ is a finitely generated projective $C^\infty(M)$-module, give us that $\Gamma$ is an equivalence.
\end{proof}

This fact was proven by Swan in 1962, \cite{swan_vector_1962}, and is known as the Serre--Swan theorem or, often in this setting, just Swan's theorem. This equivalence between categories preserves the tensor product and homomorphisms:

\begin{theorem}\cite[12.39 Theorem.]{Nes}\label{sectionstensor}
    The map
    \begin{align*}
        \Gamma(\pi)\otimes\Gamma(\eta)&\to\Gamma(\pi\otimes\eta)\\
        s\otimes t&\mapsto [p\mapsto s(p)\otimes t(p)],
    \end{align*}
    is an isomorphism for any two vector bundles $\pi$ and $\eta$ over the manifold $M$.
\end{theorem}

\begin{theorem}\cite[12.39 Theorem.]{Nes}\label{sectionshoms}
    There is an isomorphism
    \begin{align*}
        \Gamma(\Hom(\pi, \eta))&\cong\Hom_{C^\infty(M)}(\Gamma(\pi), \Gamma(\eta)).
    \end{align*}
\end{theorem}

\begin{remark}
    In fact, the above isomorphisms are natural, which makes the global sections functor $\Gamma$ a monoidal equivalence.
\end{remark}

\subsection{Localisation and Fibres}

Further investigating this equivalence we get that basic algebraic tools, such as localisation, can be realised geometrically:

\begin{proposition}\label{evalCprop}
    For every point $p\in M$, the map
    \begin{align*}
        \phi:C^\infty(M)_{\mathfrak{m}_p}/\mathfrak{m}_pC^\infty(M)_{\mathfrak{m}_p}&\longrightarrow \Lambda^0_p\\
        \left[\frac{f}{g}\right]&\longmapsto \frac{f(p)}{g(p)}
    \end{align*}
    is an isomorphism.
\end{proposition}

\begin{proof}
    The map $\phi$ is well-defined, since if $\frac{f}{g}, \frac{f'}{g'}\in C^\infty(M)_{\mathfrak{m}_p}$ are in the same class
    \begin{align*}
        \left[\frac{f}{g}\right]=\left[\frac{f'}{g'}\right]
    \end{align*}
    then, by definition, there exists an element $\frac{d}{h}\in \mathfrak{m}_pC^\infty(M)_{\mathfrak{m}_p}$ such that
    \begin{align*}
        \frac{f}{g}-\frac{f'}{g'}=\frac{d}{h}.
    \end{align*}
    Since  $\frac{d}{h}\in \mathfrak{m}_pC^\infty(M)_{\mathfrak{m}_p}$ we get that $d(p)=0$ and so
    \begin{align*}
        \phi\left[\frac{f}{g}\right]=\frac{f(p)}{g(p)}=\frac{f(p)}{g(p)}-\frac{d(p)}{h(p)}=\frac{f'(p)}{g'(p)}=\phi\left[\frac{f'}{g'}\right].
    \end{align*}
    The linearity of the map is due to the fact that evaluation is linear, i.e., $(rf+r'f')(p)=rf(p)+r'f'(p)$ for every $r, r'\in\BR$ and $f, f'\in C^\infty(M)$. Surjectivity is given by picking $f, g\in C^\infty(M)$ bump-functions about $p$, with $g(p)=1$ and $f(p)$ any value in $\BR$. Injectivity follows from
    \begin{align*}
        \phi\left[\frac{f}{g}\right]=0
    \end{align*}
    if and only if $f(p)=0$ if and only if $f\in\mathfrak{m}_p$. This concludes the proof.
\end{proof}

\begin{proposition}\label{Fibrelocalprop}
    Let $F$ be a smooth vector bundle over $M$. Then, for every point $p\in M$, the map
    \begin{align*}
        \varphi:\Gamma(F)\otimes_{C^\infty(M)}C^\infty(M)_{\mathfrak{m}_p}/\mathfrak{m}_pC^\infty(M)_{\mathfrak{m}_p}&\longrightarrow F_p\\
        [s:M\to F]\otimes\left[\frac{f}{g}\mid f, g:M\to\Lambda^0, g(p)\neq 0\right]&\longmapsto s(p)\frac{f(p)}{g(p)}
    \end{align*}
    where
    \begin{align*}
        \mathfrak{m}_p:=\{r:M\to \Lambda^0\mid\; r(p)=0\}
    \end{align*}
    is an isomorphism of vector spaces over $\BR$.
\end{proposition}

\begin{proof}
    The map is well-defined due to the map $\phi$ of proposition \ref{evalCprop} being well-defined and the fact that evaluation $s\mapsto s(p)$ is well-defined.

    The residue field $\kappa(\mathfrak{m}_p)=C^\infty(M)_{\mathfrak{m}_p}/\mathfrak{m}_pC^\infty(M)_{\mathfrak{m}_p}$ is a field isomorphic to $\BR$, by proposition \ref{evalCprop} and the fact that $\Lambda^0_p=\BR$. This means that $\Gamma(F)\otimes_{C^\infty(M)}C^\infty(M)_{\mathfrak{m}_p}/\mathfrak{m}_pC^\infty(M)_{\mathfrak{m}_p}$ is a real vector space. Since $F_p$ too is a vector space over $\BR$ we want to show that $\varphi$ is $\BR$-linear, and this is almost immediate:

    Let $r, r'\in\BR$ and $s, s'\in\Gamma(F)$ and $\left[\frac{f}{g}\right], \left[\frac{f'}{g'}\right]\in C^\infty(M)_{\mathfrak{m}_p}/\mathfrak{m}_pC^\infty(M)_{\mathfrak{m}_p}$. Then
    \begin{align*}
        \varphi\left(rs\otimes\left[\frac{f}{g}\right]+r's'\otimes\left[\frac{f'}{g'}\right]\right)&=rs(p)\frac{f(p)}{g(p)}+r's'(p)\frac{f'(p)}{g'(p)}\\
        &=r\varphi\left(s\otimes\left[\frac{f}{g}\right]\right)+r'\varphi\left(s'\otimes\left[\frac{f'}{g'}\right]\right).
    \end{align*}
    Surjectivity is, similar to before, due to the fact that we can pick $s\in\Gamma(F)$ a bump-function on a sufficiently small neighbourhood of $p$, with value $s(p)=1$. Since $s(p)\frac{f(p)}{g(p)}$ is zero whenever $s(p)=0$ or $f(p)=0$ the map is injective so we are done.
\end{proof}

\section{$G_2$-structures}\label{G2prelim}

\subsection{Short introduction to $G_2$-structures}

We again let $M$ be a connected smooth manifold and denote its dimension by $n$. Let $F(TM)$ denote the \emph{frame bundle} of the tangent bundle $TM$ of $M$. We call the elements of the fibres $F(TM)_p=\{f:\BR^n\to T_pM \text{ linear}\}$ of the frame bundle \emph{frames} and elements of its dual \emph{coframes}. It is a $\GL_n(\BR)$-principal bundle, so its rank is $n^2=\dim_\BR(\GL_n(\BR))$.

\begin{definition}[$G$-structure]\label{gstruct}
    Let $G$ be a Lie subgroup of $\GL_n(\BR)$. A \emph{$G$-structure} on $M$ is a principal $G$-bundle $F_G(TM)$ that is also a fibre subbundle of $F(TM)$, with fibrewise inclusion $G\hookrightarrow\GL_n(\BR)$.
    
    We call $F_G(TM)$ a \emph{$G$-reduction} of the frame bundle.
\end{definition}

A smooth manifold $M$ equipped with a bilinear form $g:\Gamma(TM)\otimes\Gamma(TM)\to C^\infty(M)$ that is positive-definite over every point $p\in M$ is called \emph{Riemannian}, in which case the bilinear form $g$ is called a metric. A smooth manifold $M$ admits a Riemannian structure, i.e., a metric $g$ as above, if and only if it admits an $\Orth(n)$-structure. Furthermore, a manifold $M$ is orientable if and only if it admits a $\GL^+_n(\BR)$-structure, where $\GL^+_n(\BR)\subseteq\GL_n(\BR)$ denotes the subgroup of matrices with positive determinant.

We introduce one of many definitions of the group $G_2$ as a subgroup of $\GL_7(\BR)$. The main source for this section is \cite{Bry}.

We begin by defining a 3-form $\varphi_0$ on the vector space dual of $\BR^7$ by picking a basis $e_1, \dots, e_7$ of $\BR^7$, and letting $e^1, \dots, e^7$ be the corresponding dual basis. By writing $e^{ijk}:=e^i\wedge e^j\wedge e^k$ we then define the 3-form $\varphi_0$ on $\BR^7$ as
\begin{align*}
    \varphi_0=e^{123}+e^{145}+e^{167}+e^{246}-e^{257}-e^{347}-e^{356}\in\bigwedge^3((\BR^7)^\vee).
\end{align*}
This the allows us to define the compact Lie group $G_2$ over $\BR$:

\begin{definition}[Compact group $G_2$]\cite[Definition 1]{Bry}
    We define the compact group $G_2$ as the stabilizer of $\varphi_0$, i.e.,
    \begin{align*}
        G_2:=\{\; \phi\in\GL_7(\BR): \; \phi^*(\varphi_0)=\varphi_0 \; \},
    \end{align*}
    where the pullback $\phi^*$ denotes precomposition by $\phi^{\otimes 3}$.
\end{definition}

\begin{remark}
    With $G_2$ defined, as a subgroup of $\GL_7(\BR)$, we have that a $G_2$-structure on a Riemannian manifold $M$ is a $G_2$-reduction of the frame bundle of $M$ as in definition \ref{gstruct}. Note that there is a chain of inclusions of subgroups
    \begin{align*}
        G_2\subset \SO_7(\BR)\subset \Orth_7(\BR)\subset\GL_7(\BR).
    \end{align*}
    This yields an analogous chain of $G$-reductions of the frame bundle $F(TM)$.
    
    There is an alternative characterisation of $G_2$-structure, conducted in \cite{Bry}. To make this, we first present a remark connecting to the above definition of the 3-form $\varphi_0$.
\end{remark}

\begin{remark}\cite[2.2. Associated structures]{Bry}
    The group $G_2$ preserves the metric and orientation for which $e_1, \dots, e_7$ is an oriented orthonormal basis. In this setting, a metric on a vector space is a positive-definite symmetric bilinear form. For the time being we denote the associated metric by $\langle, \rangle$. Recall that an orientation of a vector space is a choice of ordered basis together with a choice of positive sign for it. The sign of any other ordered basis is given by the sign of the determinant of the base change matrix from the chosen positive basis.

    Below we recall the definition of the \emph{Hodge star operator} associated to a metric and an orientation. The source for the following description is \cite{Jost}:

    Let $V$ be a real vector space of dimension $n$ with a positive-definite symmetric bilinear form $\langle, \rangle$ and $\bigwedge^p V$ the $p$-fold exterior product of $V$. One then obtains a scalar product on $\bigwedge^p V$ by
    \begin{align*}
        \langle v_1\wedge\dots\wedge v_p,  w_1\wedge\dots\wedge w_p\rangle=\det([\langle v_i, w_j\rangle]_{1\leq i, j\leq p})
    \end{align*}
    Letting $V$ carry an orientation, there is a linear operator
    \begin{align*}
        *:\bigwedge^pV\to \bigwedge^{n-p}V
    \end{align*}
    determined by
    \begin{align*}
        *(e_{i_1}\wedge\dots\wedge e_{i_p})=e_{j_1}\wedge\dots\wedge e_{j_{n-p}}
    \end{align*}
    where the indices $j_1, \dots, j_{n-p}$ are chosen such that $e_{i_1},\dots e_{i_p},e_{j_1},\dots, e_{j_{n-p}} $ is a positive basis of $V$.

    Returning to the setting of $\BR^7$ with the positively oriented basis $e_1, \dots, e_7$ and metric $\langle, \rangle$ we get that the Hodge star operator $*$ on $\bigwedge^7(\BR^7)^\vee$ satisfies
    \begin{align*}
        *(e^{123})=e^{4567}
    \end{align*}
    by looking at the signs of the permutations of $e_1, \dots, e_7$ (and in effect the determinants of base-change matrices). Thus, the Hodge dual of $\varphi_0$ is given by the $4$-form
    \begin{align*}
        *(\varphi_0)=e^{4567}+e^{2367}+e^{2345}+e^{1357}-e^{1346}-e^{1256}-e^{1247}.
    \end{align*}
    This means that $G_2$ also fixes this form.

    Let $\bigwedge_+^3((\BR^7)^\vee)$ be defined as the $\GL_7(\BR)$-orbit of $\varphi_0$ in $\bigwedge^3((\BR^7)^\vee)$, with the action induced by matrix multiplication. The elements of this set will all have the same associated bilinear form.
    
    With this it is possible to enter the context of a connected 7-dimensional Riemannian manifold $M$ with metric $g$, in order to give an equivalent definition of $G_2$-structures. Consider the union of the subspaces $\bigwedge_+^3(T_p^\vee M)$, with $p\in M$, as a subbundle of $\bigwedge^3(T^\vee M)$, which is denoted by $\bigwedge_+^3(T^\vee M)$.
\end{remark}

\begin{definition}[Definite 3-forms on manifolds]\cite[Definition 2]{Bry}
    A 3-form $\sigma$ on $M$ which takes values in $\bigwedge_+^3(T^\vee M)$ will be said to be a \emph{definite} 3-form on $M$.
\end{definition}

\begin{remark}
    Whenever $M$ is Riemannian, there is a canonical isomorphism\footnote{This is given by tensor-hom adjunction}
    \begin{align*}
        \Gamma\left(\bigwedge^k T^\vee M\right)\cong\bigwedge^k\Gamma(T^\vee M).
    \end{align*}
    This restricts to the positive forms, and so we will take a definite 3-form to mean an element of
    \begin{align*}
        \bigwedge^3_+\Gamma(T^\vee M).
    \end{align*}
    For the same reason we also have canonical isomorphisms
    \begin{align*}
        \bigwedge^3\Gamma(T^\vee M)\cong\bigwedge^3\Gamma(TM)^\vee\cong\left(\bigwedge^3\Gamma(T M)\right)^\vee,
    \end{align*}
    with which we may uniquely define a definite 3-form as an element of
    \begin{align*}
        \left(\bigwedge^3_+\Gamma(T M)\right)^\vee.
    \end{align*}

    Bryant explains in \cite[3.1.1. $G_2$\textit{-structures and definite 3-forms.}]{Bry} that there is a 1:1-correspondence between $G_2$-structures and definite 3-forms over a given smooth 7-dimensional manifold $M$. The explanation is given in terms of coframes, and is, as will be explained, equivalent to the following presentation of the result.
\end{remark}

\begin{proposition}\label{G2definite3form}
    Let $M$ be a smooth 7-dimensional manifold. Let $\varphi\in\bigwedge_+^3\Gamma(T^\vee M)$ be a definite 3-form. Then the fibre bundle $F_\varphi$, whose fibre over any $p\in M$ is defined by
    \begin{align*}
        (F_\varphi)_p:=\{u\in \mathrm{Isom}(\BR^7, T_pM) \; : \; u^*(\varphi_p)=\varphi_0\},
    \end{align*}
    is a $G_2$-structure over $M$.
\end{proposition}

\begin{proof}
    The bundle defined is a subbundle of the frame bundle $F(TM)$. Over any given point $p\in M$, the pullback $u^*(\varphi_p)$ is given by $\varphi_p\circ(u\otimes u\otimes u)$. We see that picking any $u\in \Hom(\BR^7, T_pM)$ satisfying $u^*(\varphi_p)=\varphi_0$, the remaining are given by pre-composition by some unique element of $G_2\subset\GL_7(\BR)$. This shows that every fibre has an action of $G_2$ that is free and transitive, which indeed gives us that $F_\varphi$ is a $G_2$-structure.
\end{proof}

\begin{remark}
    In \cite[3.1.1. $G_2$\textit{-structures and definite 3-forms.}]{Bry}, Bryant gives the analogous construction but in terms of co-frames, i.e., the definition of the bundle $F_\varphi$ is instead taken to be
    \begin{align*}
        (F_\varphi)_p:=\{u\in\mathrm{Isom}(T_pM, V) \; : \; u^*(\varphi_0)=\varphi_p\},
    \end{align*}
    where $V$ is a 7-dimensional real vector space. Note the difference in which form is being pulled back.
    
    Since there is a canonical isomorphism
    \begin{align*}
        \{u\in \mathrm{Isom}(V, T_pM) \; : \; u^*(\varphi_p)=\varphi_0\}\cong\{u\in\mathrm{Isom}(T_pM, V) \; : \; u^*(\varphi_0)=\varphi_p\}
    \end{align*}
    for any real 7-dimensional vector space $V$, the frame bundle is canonically isomorphic to the co-frame bundle, whence the construction in \cite{Bry} is dual to proposition \ref{G2definite3form}. Furthermore, Bryant states that all $G_2$-structures are given by the construction, a fact which still holds in the case of subbundles of the frame bundle and $F_\varphi$ in the notation above. This means that there is a 1:1-correspondence between $G_2$-structures on $M$ and the set $\bigwedge^3_+\Gamma(T^\vee M)$ of definite 3-forms on $M$. This can be seen by the fact that
    \begin{align*}
        (F_\varphi)_p\cong\GL(\BR)/G_2 \hspace{4mm} \text{ for every } p\in M.
    \end{align*}

    In the sequel, $G_2$-structures will be used as, or at least in reference to, their associated definite 3-forms. We continue by presenting the obstructions for a Riemannian manifold $M$ to admitting $G_2$-structure.
\end{remark}

\begin{theorem}\cite{gray_vector_1969, fernandez_riemannian_1982}\label{g2obstr}
    Let $M$ be a 7-dimensional smooth manifold. Then, $M$ admits a $G_2$-structure if and only if the Stiefel--Whitney classes $w_1$ and $w_2$ both vanish.
\end{theorem}

\begin{remark}
    The Stiefel--Whitney classes $w_i=w_i(TM)$, $i\in\BZ_{\geq 0}$, are cohomology classes of the tangent bundle $TM$ of $M$. These are classes of the cohomology ring $H^*(M;\BZ/2\BZ)$ with the cup product. The vanishing of the Stiefel--Whitney classes $w_1$ and $w_2$ correspond to $M$ being orientable and that $M$ admits a spin structure, respectively.
\end{remark}

\begin{remark}
    We now focus on a classification of $G_2$-structures that has been made. We say that $G_2$-structures with the same associated metric and orientation are \emph{isometric}. This means that two $G_2$-structures are isometric if and only if, at every point, one lies in the $\SO_7(\BR)$-orbit of the other. In the literature, a metric $g$ which admits a $G_2$-structure to begin with is often called a \emph{$G_2$-metric}.

    Grigorian presents a result originally found in \cite{Bry}. It is the precise parametrisation of isometric $G_2$-structures. Recall that the contraction $a\lrcorner \sigma$ of a $k$-form $\sigma\in\bigwedge^kT^\vee M$ by a vector field $a\in\Gamma(TM)$ constitutes evaluation in the first argument or arguments
    \begin{align*}
        a\lrcorner \sigma:=\sigma(a, -, \dots, -)\in\bigwedge^{k-1}T^\vee M.
    \end{align*}
    We then have the following result:
\end{remark}

\begin{theorem}\cite[(3.6)]{Bry}\cite[Theorem 2.3.]{Gri}\label{Bryanttwist}
    Let $M$ be a 7-dimensional manifold. Suppose $\varphi$ is a positive 3-form on $M$ with associated Riemannian metric $g$. Then any positive 3-form $\Tilde{\varphi}$ for which $g$ is also the associated metric, is given by the following expression:
    \begin{align*}
        \tilde{\varphi}=(a^2-|\alpha|^2)\varphi-2a\alpha\lrcorner(*\varphi)+2\alpha^\flat\wedge(\alpha\lrcorner\varphi).
    \end{align*}
    where $a$ is a scalar function on $M$ and $\alpha$ is a vector field such that
    \begin{align}\label{Aeq}
        a^2+g(\alpha, \alpha)=1.
    \end{align}
\end{theorem}

\begin{remark}
    Like earlier, $*\varphi$ denotes the Hodge dual of the 4-form $\varphi$. Note that changing the sign of $a$ and $\alpha$ simultaneously leaves the above expression unchanged, due to the relevant operations being linear. Later, we will consider the section $A=a+\alpha\in\Gamma(\Lambda^0\oplus TM)$ with $\Lambda^0$ the trivial bundle over $M$.

    We denote a $G_2$-structure such as the above by $\tilde{\varphi}_A$, where $A$ denotes the pair $(a, \alpha)$ satisfying equation \ref{Aeq}. In light of this, Grigorian remarks that, due to the theorem above, the $G_2$-structures on $M$ are parametrised by sections of the $\BR\BP^7$-bundle $\BP(\Lambda^0\oplus TM)$. This can be verified by the fact that
    \begin{align*}
        \SO_7(\BR)/G_2\cong\BR\BP^7
    \end{align*}
    locally parametrises isometric $G_2$-structures.
\end{remark}

\subsection{The category of $G_2$-structures}

We introduce a definition of the category of $G_2$-structures. Although we will again focus on $G_2$-structures as definite 3-forms, the definition will be natural in light of the classical definition of a $G_2$-structure. In particular, we present two definitions, one for each notion of $G_2$-structure, and show that these agree up to isomorphism of categories. In order to define the morphisms of these respective categories we present two respective definitions of pushforward of a $G$-structure.

\begin{definition}[Pushforward of $G$-structure]
    Let $P\overset{\pi}{\to} M$ be a $G$-structure viewed as a $G$-principal subbundle of $F(TM)$, with inclusion $P\overset{\iota}{\hookrightarrow}F(TM)$. Let $\mathrm{f}:TM\overset{\cong}{\to}TM$ be an automorphism of the tangent bundle $TM$. We then define the \emph{pushforward of $P$ along $\mathrm{f}$}, denoted by $\mathrm{f}P$, as the $G$-principal by $P$ and as a subbundle of $F(TM)$ by the composition
    \begin{align*}
        P\overset{\iota}{\hookrightarrow} F(TM)\overset{\mathrm{f}\circ}{\to} F(TM).
    \end{align*}
\end{definition}

\begin{remark}
    The pushforward $\mathrm{f}P$ has the structure of a $G$-principal bundle as well as a subbundle of $F(TM)$, as $F(TM)\overset{f\circ}{\to} F(TM)$ is an isomorphism of $\GL_n(\BR)$-principal bundles. This means that $\mathrm{f}P$ is a $G$-structure.

    For the remainder of this subsection we fix a connected 7-dimensional Riemannian manifold $M$ with metric $g$.
\end{remark}

\begin{definition}[Category of $G_2$-structures (principal bundles)]\label{g2catprinc}
    Let the category denoted by $\catname{G2str}(M)$ consist of the following data:
    \begin{itemize}
        \item Objects: $G_2$-structures viewed as $G_2$-reductions of the frame bundle $F(TM)$.
        \item Morphisms: Automorphisms $\mathrm{f}:TM\overset{\cong}{\to}TM$ such that $P'=\mathrm{f}P$. We will, by abuse of notation, write this as $\mathrm{f}:P\to P'$.
    \end{itemize}
    We refer to it as the \emph{category of $G_2$-structures}.
\end{definition}

\begin{definition}[Pushforward of 3-forms]
    Let $\varphi\in\bigwedge_+^3\Gamma(T^\vee M)$ be a $G_2$-structure viewed as a definite 3-form. Let $f:\Gamma(TM)\overset{\cong}{\to}\Gamma(TM)$ be a $C^\infty(M)$-linear automorphism of the module $\Gamma(TM)$. We then define the \emph{pushforward of $\varphi$}, denoted by $f_*\varphi$, as the 3-form given by the composition
    \begin{align*}
        \varphi\circ(f^{-1}\otimes f^{-1}\otimes f^{-1}).
    \end{align*}
\end{definition}

\begin{remark}
    Since $f$ is an isomorphism we get that it is locally given by isomorphisms $f_p:T_pM\overset{\cong}{\to}T_pM$. Thus, $f_*\varphi$ is, over every point $p\in M$, contained in the $\GL_7(\BR)$-orbit of $\varphi_p$ and hence in the orbit of $\varphi_0$, by virtue of $\varphi$ being a definite 3-form. This means that $f_*\varphi$ itself is a definite 3-form. 
\end{remark}

\begin{definition}[Category of definite 3-forms]\label{g2cat3form}
    Let the category denoted by $\catname{def3form}(M)$ consist of the following data:
    \begin{itemize}
        \item Objects: $G_2$-structures as definite 3-forms. In other words,
        \begin{align*}
            \mathrm{ob}(\catname{def3form}(M))=\bigwedge_+^3\Gamma(TM)
        \end{align*}
        \item Morphisms: $C^\infty(M)$-linear automorphisms $f:\Gamma(TM)\overset{\cong}{\to}\Gamma(TM)$ such that $\varphi'=f_*\varphi$. We write this as $f:\varphi\to \varphi'$.
    \end{itemize}
    We refer to this as the \emph{category of definite 3-forms}.
\end{definition}

\begin{remark}
    We may refer to both categories as \emph{category of $G_2$-structures} due to the following facts.
\end{remark}

\begin{lemma}\label{G2catdefcoh}
    Let $\varphi$ be a definite 3-form with associated $G_2$-principal bundle $F_\varphi$. Also let $\mathsf{f}:TM\overset{\cong}{\to}TM$ be an isomorphism of smooth vector bundles and $f=\Gamma(\mathsf{f})$ the associated $C^\infty(M)$-module automorphism of $\Gamma(TM)$. We then have that
    \begin{align*}
        \mathsf{f}F_{\varphi}=F_{f_*\varphi}.
    \end{align*}
\end{lemma}

\begin{proof}
    By direct calculation we get that
    \begin{align*}
        F_{f_*\varphi}&=\{u:\BR^7\to T_pM\; \mid \; p\in M\; \text{ and }\; \varphi_0=u^*(f(p)_*\varphi_p)\}\\
        &=\{u:\BR^7\to T_pM\; \mid \; p\in M\; \text{ and }\; \varphi_0=(f(p)^{-1}\circ u)^*(\varphi_p)\}\\
        &=\{f(p)\circ w:\BR^7\to T_pM\; \mid \; p\in M\; \text{ and }\; \varphi_0=w^*(\varphi_p)\}\\
        &=\{\mathsf{f}_p\circ w:\BR^7\to T_pM\; \mid \; p\in M\; \text{ and }\; \varphi_0=w^*(\varphi_p)\}\\
        &=\mathsf{f}F_\varphi
    \end{align*}
    which proves the assertion.
\end{proof}

\begin{proposition}[Isomorphism of the above categories]\label{G2catequiv}
    There is an isomorphism of categories
    \begin{align*}
        \catname{def3form}(M)&\overset{\cong}{\longrightarrow}\catname{G2str}(M)\\
        \varphi&\longmapsto F_\varphi
    \end{align*}
\end{proposition}

\begin{proof}
    The remark of Bryant shows that we have a 1:1-correspondence of the objects. Lemma \ref{G2catdefcoh} gives us that there is an isomorphism
    \begin{align*}
        \Hom_{\catname{def3form}(M)}(\varphi, \varphi')&=\{f:f_*\varphi=\varphi'\}\\
        &=\{f:F_{f_*\varphi}=F_{\varphi'}\}\hspace{6mm} \text{$\Gamma$ is an equivalence}\\
        &\cong\{\mathsf{f}:\mathsf{f}F_\varphi=F_{\varphi'}\}\\
        &=\Hom_{\catname{G2str}(M)}(F_\varphi, F_{\varphi'}).
    \end{align*}
    This concludes the proof.
\end{proof}

\begin{remark}
    Due to the above isomorphism between the two categories we will be using $\catname{G2str}(M)$ as the preferred reference for the category of $G_2$-structures.
\end{remark}

\section{Octonion algebras}\label{octalgprelim}

We present an overview of octonion algebras, both over fields and over arbitrary rings. Much of what is presented, and used later in the article, that concerns the former is gathered from \cite{SV}, while the latter references the articles \cite{LPR} and \cite{AlGi}.

\subsection{Octonion algebras over fields}

We begin by recalling the definition of a composition algebra. For the moment, let $k$ denote an arbitrary field.

\begin{definition}
    A \emph{composition algebra} $C$ over $k$ is a not necessarily associative algebra over $k$ with identity element $e$ such that there exists a quadratic form $q$ on C which satisfies
    \begin{align*}
        q(xy) = q(x)q(y) \hspace{6mm} (x, y\in C).
    \end{align*}
    and that the induced bilinear form $b_q$, defined by
    \begin{align*}
        b_q(x, y)=q(x+y)-q(x)-q(y).
    \end{align*}
    is non-degenerate.
\end{definition}

\begin{remark}\label{uniqueoctalg/field}
    We call a composition algebra $C$ that has rank 8 over $R$ an \emph{octonion algebra} over $R$.

    Octonion algebras over fields are more well-behaved than octonion algebras over general rings. In particular, any composition algebra over a field is completely determined by its bilinear form. There exists only one isomorphism class of octonion algebras over the complex numbers $\BC$. The situation over $\BR$ is different in that there are two possible isomorphism classes of octonion algebras; the members of one isomorphism class are called the split octonions over $\BR$, containing zero divisors, and the other is referred to as the non-split octonions over $\BR$, which does not contain any zero divisors.

    An exposition on octonion algebras over fields, and the more general concept of composition algebras over fields, is given in \cite{SV}. We will not state the important results of this source here, but rather refer to those results when needed later.
\end{remark}

\begin{remark}\label{Pfisterformintro}
    There is a characterisation of the quadratic forms $q:C\to \BR$ for which the pair $(C, q)$ admits an octonion multiplication $m$, making the triple $(C, q, m)$ into an octonion algebra over $\BR$, \cite[1.10 Classification over Special Fields]{SV}. These are the \emph{(threefold) Pfister forms} and are given, under a choice of basis $\{e_0, \dots, e_7\}$, by
    \begin{align*}
        e_0^2+\alpha e_1^2+\beta e_2^2+\alpha\beta e_3^2+\gamma e_4^2+\alpha\gamma e_5^2+\beta\gamma e_6^2+\alpha\beta\gamma e_7^2
    \end{align*}
    with non-zero constants $\alpha, \beta, \gamma\in\BR$.
\end{remark}

\subsection{Octonion algebras over arbitrary rings}

This section briefly presents some results on octonion algebras over commutative rings. The main sources for this section are \cite{LPR} and \cite{AlGi}. Throughout we let $k$ denote an arbitrary commutative ring and we assume an algebra over $k$.

\begin{definition}\cite[4.1. The concept of an octonion algebra]{LPR}\label{octalgdef}
    An algebra $C$ over $k$ is called an \emph{octonion algebra} if it is finitely generated and projective of rank 8 as a $k$-module, contains an identity element and admits a norm, i.e. a quadratic form $q_C:C\to k$, uniquely determined by the following two conditions:
    \begin{enumerate}[label=(\roman*)]
        \item The induced symmetric bilinear form
        \begin{align*}
            b_C(x, y)=q_C(x+y)-q_C(x)-q_C(y)
        \end{align*}
        defines a linear isomorphism from the $k$-module $C$ onto its dual $C^\vee$ by the assignment $x\mapsto b_C(x, -)$.

        \item $q_C$ permits composition, i.e., the relation
        \begin{align}\label{multnorm}
            q_C(xy)=q_C(x)q_C(y)
        \end{align}
        holds for all $x, y\in C$.
    \end{enumerate}
\end{definition}

\begin{remark}
    We present fact which unlocks a possible variant of the above definition over a ring $R$ containing $\frac{1}{2}$. In particular, it makes no mention of a quadratic form. It will be used in defining octonion algebra objects, and when doing so we prefer (multi-)linear maps to quadratic forms.
\end{remark}

\begin{lemma}\label{octalglemma}
    Let $R$ be a ring containing $\frac{1}{2}$. Then an $R$-algebra $A$ over $R$ with a quadratic form $q_A$ with associated symmetric bilinear form $b_A$ satisfies
    \begin{align*}
        q_A(ab)=q_A(a)q_A(b)
    \end{align*}
    for every $a, b\in A$ if and only if it satisfies
    \begin{align}\label{bilformmultcomm}
        b_A(x_1, x_2)b_A(y_1, y_2)=b_A(x_1y_1, x_2y_2)+b_A(x_1y_2, x_2y_1)
    \end{align}
    for every $x_1, x_2, y_1, y_2\in A$.
\end{lemma}

\begin{proof}
    The definition \ref{octalgdef} implies that the relation \ref{bilformmultcomm} is satisfied; this direction is proven identically to \cite[5]{SV}.

    For the other direction, we assume that $A$ is an $R$-algebra with (non-degenerate) symmetric bilinear form $b_A$ satisfying \ref{bilformmultcomm}. Then
    \begin{align*}
        b_A(a, a)b_A(b, b)&=b_A(ab, ab)+b_A(ab, ab)\\
        &=2b_A(ab, ab).
    \end{align*}
    This means that the associated quadratic form $q_A$, defined by $2q_A(x)=b_A(x, x)$ for $x\in A$, satisfies
    \begin{align*}
        2q_A(a)2q_A(b)=2(2q_A(ab))\overset{\frac{1}{2}\in R}{\iff} q_A(a)q_A(b)=q_A(ab).
    \end{align*}
    This concludes the proof.
\end{proof}
We now return to the context of a general ring $k$ and continue by introducing some notions of maps of octonion algebras.
\begin{definition}[Isometry]
    Let $C$ be a $k$-module and let $q$ be a quadratic form over $C$, i.e., a map (of sets) $q:C\to k$ such that $q(\lambda m)=\lambda^2q(m)$ for every $\lambda\in k$ and every $m\in C$. We call any such pair $(C, q)$ a \emph{quadratic module}.

    For two quadratic modules $(C, q)$ and $(C', q')$, an \emph{isometry} from $(C, q)$ to $(C', q')$ is a $k$-module homomorphism
    \begin{align*}
        \phi:C\to C'
    \end{align*}
    such that $q'\circ\phi=q$. If $\phi$ is as above we say that $\phi$ is an \emph{isometry}, hence $(C, q)$ and $(C', q')$ are in that case said to be \emph{isometric}.
\end{definition}

We now introduce some ways of comparing the multiplications of algebras.

\begin{definition}[Isotopy]\label{isotopedef}
    Let $(A, *_A)$ and $(B, *_B)$ be $k$-algebras. They are said to be \emph{isotopic} if there exist $k$-module isomorphisms $f_i:A\to B$, $i=1, 2, 3$, such that
    \begin{align*}
        f_1(x*_A y)=f_2(x)*_Bf_3(y)
    \end{align*}
    for every $x, y\in A$.
\end{definition}

\begin{remark}
    When the maps $f_i=:f$ for every $i=1, 2, 3$ in the above definition, we see that $f$ is an isomorphism of $k$-algebras. If $f:A\to B$ is an isomorphism of $k$-algebras that are also octonion algebras, then $f$ is an isometry when considered as a map of quadratic modules. Two $k$-algebras are said to be \emph{isotopic} or, alternatively, \emph{isotopes} if there exists an isotopy between them.

    The paper \cite{AlGi} provides a complete parametrisation of the octonion algebras over $k$ with isometric quadratic forms. The parametrisation may be formulated as follows:
\end{remark}

\begin{theorem}\label{twistoctalg}
    Let $C$ be an octonion algebra over $k$ and $q$ its quadratic form. Let $*$ denote its multiplication. Then, if $C'$ is an octonion algebra over $k$ with $(C', q')$ isometric to $(C, q)$, there is isomorphism
    \begin{align*}
        C'\cong C^{a, b}
    \end{align*}
    where $C^{a, b}$ has the same underlying module as $C$ and multiplication $*'$ given by
    \begin{align*}
        x*'y=(x*a)*(b*y)
    \end{align*}
    for $a, b\in C$ satisfying $q(a)=q(b)=1$.
\end{theorem}

\begin{remark}
    As an isotope of $C$, an isotopy $C\to C^{a, b}$ is, with the notation in definition \ref{isotopedef}, defined by $f_1=\id$, $f_2=R_a$ and $f_3=L_b$, with the latter two maps being right- and left multiplication respectively.
\end{remark}

\begin{remark}
    This is related to the quotient group scheme $\mathrm{Spin}(q_C)/\Aut(C)\cong\BS^7$, with $\mathrm{Spin}(q_C)$ the double cover of $\SO(q_C)$ and $\Aut(C)$ the automorphism $k$-group scheme of $C$, which is semisimple and of type $G_2$, appears in the proof of theorem \ref{twistoctalg}, and relates to the Lie group quotient
    \begin{align*}
        \SO_7(\BR)/G_2\cong\BR\BP^7
    \end{align*}
    which will play an important role later.
    
    The relations between the isotopes $C^{a, b}$ appearing in \ref{twistoctalg} are studied in \cite[Proposition 2.4.]{AlGi}. Using isomorphisms appearing in the proposition, theorem \ref{twistoctalg} can be reduced further to any octonion algebra $C'$ isometric to $C$ being isomorphic to $C^{a, a^{-1}}$ where $q_C(a)=1$. In particular,
    \begin{align*}
        C^{a, b}&\cong C^{1, ba^{-1}} &&\text{\cite[Proposition 2.4. (3)]{AlGi}}\\
        &\cong C^{1, q(ba^{-1})(ba^{-1})^{-1}}&&\text{multiplication by scalar is linear}\\
        &\cong C^{q(ba^{-1})(ba^{-1})^{-1}, q(ba^{-1})^{-1}ba^{-1}}&&\text{\cite[Proposition 2.4. (2)]{AlGi}}.
    \end{align*}
\end{remark}

\subsection{The category of octonion algebras}

We define the category of octonion algebras, with the category $\proj_R$ as a starting point. Since an octonion algebra $C$ over a commutative ring $R$ is required to be a projective module of rank 8, morphisms of octonion algebras are required to be linear isomorphisms of the base module in order to preserve this structure.

\begin{definition}\label{octalgpushf}
    Let $C$ and $C'$ be rank 8 projective modules over a ring $R$ and let $f:C\to C'$ be an $R$-linear isomorphism, i.e., an isomorphism of $R$-modules. If we let $C$ be an octonion algebra, and let $m_C, e_C, b_C$ be the multiplication, identity respectively bilinear form, we define the multiplication, identity and bilinear form on the target by the following
    \begin{align*}
        (f_*m_C)(x\otimes y)&:=f(m(f^{-1}(x)\otimes f^{-1}(y))),\\
        (f_*e)&:=f(e),\\
        (f_*b_C)(x\otimes y)&:=b_C(f^{-1}(x)\otimes f^{-1}(y)).
    \end{align*}
\end{definition}

\begin{definition}[Category of octonion algebras over $R$]\label{octalgcatdef}
    Let $R$ be a commutative ring. Let the category $\Octalg_R$ consist of the following data:
    \begin{itemize}
        \item Objects: Octonion algebras over $R$ as in definition \ref{octalgdef}.
        \item Morphisms: Let $C$ and $C'$ be octonion algebras over $R$ with $m_C, e_C, b_C$ and $m_{C'}, e_{C'}, b_{C'}$ the multiplication, identity and bilinear form of $C$ and $C'$ respectively. An isomorphism of $R$-modules
        \begin{align*}
            f:C\overset{\cong}{\to} C'
        \end{align*}
        is a \emph{morphism of octonion algebras} if it satisfies the following:
        \begin{align}\label{octalgmorphrels}
            m_{C'}&=f_*m_C,\nonumber\\
            e_{C'}&=f_*e_C,\\
            b_{C'}&=f_*b_C.\nonumber
        \end{align}
    \end{itemize}
\end{definition}

\begin{remark}
    The second condition of \ref{octalgmorphrels} follows from the first condition, since the unit is unique and since the map $f:C\to C'$ is bijective. Due to there being an inclusion $\mathrm{Span}_R\{e_C\}\hookrightarrow C$ for any octonion algebra $C$ and since the unit is preserved under morphisms of octonion algebras, this means that such a morphism is only determined by its values on and in the parts of $C$ and $C'$ that do not belong to the spans of the respective identities.

    By definition, this is a category in which every morphism is an isomorphism. This makes $\Octalg_R$ a groupoid.
\end{remark}

\section{Octonion bundles}\label{obsect}

In this section we pave the way for describing $G_2$-structures as octonion algebras. In the first subsection we define octonion algebra objects in $\Vecbun(M)$. They are constructed in such a way that, by taking the sections of their defining data, they correspond to octonion algebras over $C^\infty(M)$. We use this in the second subsection to define octonion bundles, which are octonion algebra objects of the category of vector bundles over $M$ which we will show correspond to $G_2$-structures.

\subsection{Octonion algebra objects in $\Vecbun(M)$}\label{octalgobjsection}

Before presenting the definition of an octonion algebra object, let us, for any vector bundle $F$, define $\mathsf{T}_{1324}$ and $\mathsf{T}_{1423}$ as the twist maps
\begin{align*}
    \mathsf{T}_{1324}:F\otimes F\otimes F\otimes F&\longrightarrow F\otimes F\otimes F\otimes F\\
    x_1\otimes x_2\otimes x_3\otimes x_4 &\longmapsto x_1\otimes x_3\otimes x_2\otimes x_4
\end{align*}
and
\begin{align*}
    \mathsf{T}_{1423}:F\otimes F\otimes F\otimes F&\longrightarrow F\otimes F\otimes F\otimes F\\
    x_1\otimes x_2\otimes x_3\otimes x_4 &\longmapsto x_1\otimes x_4\otimes x_2\otimes x_3.
\end{align*}
Note that these are isomorphisms of vector bundles. Also recall that the sum of two morphisms of vector bundles is itself a morphism of vector bundles.

\begin{definition}[Octonion algebra object]\label{octalgobjdef}
    Let $M$ be a smooth manifold and let $\Lambda^0$ denote the trivial bundle. We define an \emph{octonion algebra object} (in the category of vector bundles over $M$) to be a smooth vector bundle $F\overset{\pi}{\to} M$ of rank 8 together with morphisms of vector bundles $\mathsf{m}:F\otimes F\to F$, $\mathsf{e}:\Lambda^0\to F$ and $\mathsf{b}:F\otimes F\to \Lambda^0$ which make the following diagrams commute:
    \begin{center}
        \begin{tikzcd}
        \Lambda^0\otimes F \arrow[r, "\mathsf{e}\otimes\mathrm{Id}"] \arrow[rd, "\cong"'] & F\otimes F \arrow[d, "\mathsf{m}"] & F\otimes \Lambda^0 \arrow[l, "\mathrm{Id}\otimes \mathsf{e}"'] \arrow[ld, "\cong"] \\
                                                                                   & F                         &                                                                            
        \end{tikzcd}
        \hspace{8mm}
        \begin{tikzcd}
        F\otimes F \arrow[r, "\mathsf{b}"] \arrow[d, "\mathsf{T}"'] & \Lambda^0 \\
        F\otimes F \arrow[ru, "\mathsf{b}"']               &     
        \end{tikzcd}
        \end{center}
        \begin{center}
        \begin{tikzcd}
        F\otimes F\otimes F\otimes F \arrow[d, "\mathsf{b}\otimes \mathsf{b}"'] \arrow[r, "{\lambda_{(\mathsf{m}, \mathsf{b})}}"] & \Lambda^0 \\
        \Lambda^0\otimes\Lambda^0 \arrow[ru, "\cong"']                                                                            &          
        \end{tikzcd}
    \end{center}
    Here, $\lambda_{(\mathsf{m}, \mathsf{b})}$ denotes the sum of morphisms
    \begin{align*}
        \mathsf{b}\circ(\mathsf{m}\otimes\mathsf{m})\circ \mathsf{T}_{1324}+\mathsf{b}\circ(\mathsf{m}\otimes\mathsf{m})\circ \mathsf{T}_{1423}.
    \end{align*}
    Further, we require for $\mathsf{b}$ to be nowhere degenerate, i.e., the map $v\mapsto \mathsf{b}_p(v, -)$ from $F_p$ to $F_p^\vee$ is an isomorphism of vector spaces for every point $p\in M$.
\end{definition}

\begin{remark}
    This way of defining an octonion algebra in terms of commuting diagrams has been done for instance in \cite{daza-garcia_octonions_2024}, in this case to construct composition algebra objects in the category of super vector spaces. Indeed, the above definition works in any symmetric tensor category (over a field of characteristic not 2), but also in the category $\proj_R$ for any ring $R\ni\frac{1}{2}$, since we have a well defined notion of rank, tensor product, twist maps and translation between quadratic and bilinear forms.
\end{remark}

Before presenting the following proposition, we provide the following lemma:

\begin{lemma}\label{localdegen=sectiondegen}
    Let $M$ be a smooth manifold and let $F$ be a vector bundle over $M$. Let $\mathsf{b}:F\otimes F\to \Lambda^0$ be a morphism of vector bundles. Then $\mathsf{b}_p:F_p\otimes F_p\to \Lambda^0_p=\BR$ is non-degenerate for every point $p\in M$ if and only if the morphism $b:\Gamma(F)\otimes\Gamma(F)\to \Gamma(\Lambda^0)\cong C^\infty(M)$ given by $b=\Gamma(\mathsf{b})$ is non-degenerate as a bilinear form of $C^\infty(M)$-modules.
\end{lemma}

\begin{proof}
    Denote by $P=\Gamma(F)$, the sections of $F$ over $M$. The map $b:P\otimes P\to C^\infty(M)$ is induced by the sections of the map $\mathsf{b}$.
    
    We have $\Gamma(\mathsf{b})\in\Gamma(\Hom(F\otimes F, \Lambda^0))$. Using the second natural isomorphism of theorem \ref{sectionstensor} we can identify $\Gamma(\mathsf{b})$ with an element of $\Hom(\Gamma(F\otimes F), \Gamma(\Lambda^0))$, namely the map $b'$ defined by
    \begin{align*}
        b':\Gamma(F\otimes F)&\to \Gamma(\Lambda^0)\\
        [s:M\to F\otimes F] &\mapsto b'(s)
    \end{align*}
    where $b'(s):M\to \Lambda^0$ sends a point $p\in M$ to $\Gamma(\mathsf{b})(p)\circ s(p)$. In turn, using the natural isomorphism of theorem \ref{sectionstensor}, we get that the map $b$ is given by sending a section $s\otimes t\in P\otimes P$ to the section
    \begin{align}\label{bdef}
        M&\overset{b(s\otimes t)}{\longrightarrow} \Lambda^0\\ \nonumber
        p&\longmapsto \Gamma(\mathsf{b})(p)\circ(s(p)\otimes t(p)).
    \end{align}
    Note that there is equality $\Gamma(\mathsf{b})(p)=\mathsf{b}_p$ for $p\in M$.

    The bilinear form $b:P\otimes P\to \Gamma(\Lambda^0)$ is non-degenerate if and only if the map
    \begin{align*}
        P&\to P^\vee\\
        s&\mapsto b(s, -)
    \end{align*}
    is an isomorphism. Since there is an isomorphism
    \begin{align*}
        P^\vee=\Gamma(F)^\vee=\Hom(\Gamma(F), \Gamma(\Lambda^0))\cong\Gamma(\Hom(F, \Lambda^0))=\Gamma(F^\vee)
    \end{align*}
    the above map $P\to P^\vee$ being an isomorphism is equivalent to the induced map
    \begin{align*}
        P&\to \Gamma(F^\vee)
    \end{align*}
    being an isomorphism. This map sends a section $s\in P$ to the section of $\Hom(F, \Lambda^0)$ which sends the point $p\in M$ to the map
    \begin{align*}
        F_p&\to \Lambda^0_p\\
        v&\mapsto (b(s, -))(t)(p)=b(s, t)(p),
    \end{align*}
    where $t\in\Gamma(F)$ is a section chosen to satisfy $t(p)=v$ for $v\in F_p$, making the above map well-defined. Such sections always exist thanks to the existence of bump-functions. Again using the natural isomorphism to identify the map $\Gamma(F)\to \Gamma(F^\vee)$ with a section in $\Gamma(\Hom(F, F^\vee))$, we see that this resulting section is defined by sending a point $p\in M$ to
    \begin{align*}
        F_p&\to F_p^\vee\\
        v&\mapsto [w\mapsto b(s, t)(p)]
    \end{align*}
    by choosing for $v, w\in F_p$ sections $s, t\in P$ with $s(p)=v$ and $t(p)=
    w$. Recalling from the expression \ref{bdef} that $b(s, t)(p)=\mathsf{b}_p(s(p)\otimes t(p))$, this tells us that this is the section of the map of bundles $F\to F^\vee$ given for every point $p\in M$ by
    \begin{align}\label{localmap}
        F_p&\to F_p^\vee \\
        v&\mapsto [w\mapsto \mathsf{b}_p(v, w)]=\mathsf{b}_p(v, -). \nonumber
    \end{align}
    Since $\Gamma$ is an equivalence of categories, this morphism of bundles is an isomorphism if and only if the map $P\to P^\vee$, $s\mapsto b(s, -)$, is an isomorphism. The bundle map $F\to F^\vee$ is an isomorphism if and only if for every point $p\in M$ the map \ref{localmap} is an isomorphism. This concludes the proof.
\end{proof}

We now show that the datum of an octonion algebra object is the same as the datum of an octonion algebra over $C^\infty(M)$. We later upgrade this to an equivalence of categories.

\begin{proposition}\label{objalgcorrprop}
    Octonion algebra objects over $M$ correspond 1:1 to octonion algebras over $C^\infty(M)$ by the global sections functor $\Gamma$.
\end{proposition}

\begin{proof}
    The module $P=\Gamma(F)$ is a projective module of rank 8 over $C^\infty(M)$ if and only if $F$ is a rank 8 vector bundle over $M$. A bilinear form $b:P\otimes P\to \Lambda^0$ is non-degenerate if and only if every $\mathsf{b}_p:F_p\otimes F_p\to \Lambda^0_p$ is non-degenerate and this follows from lemma \ref{localdegen=sectiondegen}. Note that the latter is part of definition \ref{octalgobjdef}, so the non-degeneracy statements are equivalent. The natural isomorphism of \ref{sectionstensor} ensures that the map
    \begin{align*}
        \Gamma(\mathsf{T}):\Gamma(F\otimes F)&\longrightarrow\Gamma(F\otimes F)\\
        [s:p\mapsto s_1(p)\otimes s_2(p)]&\longmapsto[\Gamma(\mathsf{T})(s):p\mapsto s_2(p)\otimes s_1(p)]
    \end{align*}
    is naturally isomorphic to the map 
    \begin{align*}
        T:\Gamma(F)\otimes\Gamma(F)&\longrightarrow\Gamma(F)\otimes\Gamma(F)\\
        s_1\otimes s_2&\longmapsto s_2\otimes s_1.
    \end{align*}
    By functoriality, and theorem \ref{sectionstensor}, this means that the twist maps $\mathsf{T}_{1324}$ and $\mathsf{T}_{1423}$ are sent to the maps
    \begin{align*}
        T_{1324}:P\otimes P\otimes P\otimes P&\overset{\cong}{\longrightarrow} P\otimes P\otimes P\otimes P\\
        x_1\otimes x_2\otimes x_3\otimes x_4 &\longmapsto x_1\otimes x_3\otimes x_2\otimes x_4
    \end{align*}
    and
    \begin{align*}
        T_{1423}:P\otimes P\otimes P\otimes P&\overset{\cong}{\longrightarrow} P\otimes P\otimes P\otimes P\\
        x_1\otimes x_2\otimes x_3\otimes x_4 &\longmapsto x_1\otimes x_4\otimes x_2\otimes x_3,
    \end{align*}
    in particular since $\mathsf{T}_{1324}=\Id_1\otimes \mathsf{T}\otimes \Id_4$ and $\mathsf{T}_{1423}=(\Id_1\otimes\mathsf{T}\otimes\Id_4)\circ(\Id_{12}\otimes \mathsf{T})$. The natural isomorphism $\Gamma(F\otimes F)\cong\Gamma(F)\otimes\Gamma(F)$ similarly gives us maps
    \begin{align*}
        m:P\otimes P\to P, \hspace{6mm} b:P\otimes P\to C^\infty(M), \hspace{6mm} T:P\otimes P\to P\otimes P
    \end{align*}
    in the category $\proj_{C^\infty(M)}$, corresponding to $\Gamma(\mathsf{m})$, $\Gamma(\mathsf{b})$ and $\Gamma(\mathsf{T})$. Due to $\Gamma$ being a functor of additive categories, this in turn means that $\lambda_{(\mathsf{m}, \mathsf{b})}$ is sent to the map
    \begin{align*}
        \lambda_{m, b}:=b\circ(m\otimes m)\circ T_{1324}+b\circ(m\otimes m)\circ T_{1423}.
    \end{align*}
    Finally, since $\Gamma$ is a functor, the maps given by these identifications satisfy
    \begin{center}
        \begin{tikzcd}
        C^\infty(M)\otimes P \arrow[r, "e\otimes\mathrm{Id}"] \arrow[rd, "\cong"'] & P\otimes P \arrow[d, "m"] & P\otimes C^\infty(M) \arrow[l, "\mathrm{Id}\otimes e"'] \arrow[ld, "\cong"] \\
                                                                                   & P                         &                                                                            
        \end{tikzcd}
        \hspace{8mm}
        \begin{tikzcd}
        P\otimes P \arrow[r, "b"] \arrow[d, "T"'] & C^\infty(M) \\
        P\otimes P \arrow[ru, "b"']               &     
        \end{tikzcd}
        \end{center}
        \begin{center}
        \begin{tikzcd}
        P\otimes P\otimes P\otimes P \arrow[d, "b\otimes b"'] \arrow[r, "{\lambda_{(m, b)}}"] & C^\infty(M) \\
        C^\infty(M)\otimes C^\infty(M) \arrow[ru, "\cong"']                                                                            &          
        \end{tikzcd}
    \end{center}
    where, again, $\lambda_{m, b}:=b\circ(m\otimes m)\circ T_{1324}+b\circ(m\otimes m)\circ T_{1423}$.
    
    The map $m$ determines a multiplication on $P$ with unit given by $e$, by the upper left diagram. Furthermore, $P$ admits a bilinear form $b$, by the fact that it is defined on the tensor product $P\otimes P$, which is symmetric by the top right diagram. The lower diagram means that for $x_1, x_2, y_1, y_2\in P$,
    \begin{align*}
        b(x_1, x_2)b(y_1, y_2)=b(x_1y_1, x_2y_2)+b(x_1y_2, x_2y_1).
    \end{align*}
    By lemma \ref{octalglemma}, this is the second condition of definition \ref{octalgdef}. Note that these relations are not only implied by but equivalent to the corresponding diagrams commuting in definition \ref{octalgobjdef}; we see this from evaluating the maps on elements. Thus, using the functor $\Gamma^{-1}$ for the other direction analogously, we get that there is a 1:1-correspondence between octonion algebras over $C^\infty(M)$ and octonion algebra objects over $\Vecbun(M)$.
\end{proof}

This allows us to define morphisms of octonion algebra objects in the same way as morphisms of octonion algebras over a ring, and this updates the correspondence to an equivalence between the category of octonion algebra objects over $M$ and the category of octonion algebras over $C^\infty(M)$.

\begin{definition}[Category of octonion algebra objects over $\Vecbun(M)$]\label{octalgobjcatdef}
    Let $M$ be a smooth manifold and let $\Vecbun(M)$ be its category of smooth vector bundles. We let the category $\catname{octobj}_{\Vecbun(M)}$ consist of the following data:
    \begin{itemize}
        \item Objects: Octonion algebra objects over $M$ as defined in \ref{octalgobjdef}.
        
        \item Morphisms: Let $(F, \mathsf{m}, \mathsf{e}, \mathsf{b})$ and $(F', \mathsf{m}', \mathsf{e}', \mathsf{b}')$ be octonion algebra objects in $\Vecbun(M)$. An isomorphism
        \begin{align*}
            f:F\overset{\cong}{\to} F'
        \end{align*}
        is a \emph{morphism of composition algebra objects} if it satisfies the following:
        \begin{align}\label{octalgobjmorphrels}
            \mathsf{m'}&=f_*\mathsf{m},\nonumber\\
            \mathsf{e'}&=f_*\mathsf{e},\\
            \mathsf{b'}&=f_*\mathsf{b}.\nonumber
        \end{align}
    \end{itemize}
\end{definition}

\begin{remark}
    The pushforward is defined analogously to \ref{octalgpushf}, i.e., the relations \ref{octalgobjmorphrels} can be written as
    \begin{align*}
        \mathsf{m'}&=f\circ\mathsf{m}\circ(f^{-1}\otimes f^{-1}),\\
        \mathsf{e'}&=f\circ\mathsf{e}, \\
        \mathsf{b'}&=\mathsf{b}\circ(f^{-1}\otimes f^{-1}).
    \end{align*}
    Since $f:F\to F'$ is assumed to be an isomorphism, $\catname{octobj}_{\Vecbun(M)}$ is a groupoid.

    We can now prove that the category of octonion algebra objects over $M$ is equivalent to the octonion algebras over $C^\infty(M)$:
\end{remark}

\begin{proposition}[Equivalence]\label{octalgcatobjequivalence}
    The functor
    \begin{align*}
        \Octalgobj_{\Vecbun(M)}&\overset{\Gamma}{\longrightarrow} \Octalg_{C^\infty(M)}\\
        (F, \mathsf{m}, \mathsf{e}, \mathsf{b})&\longmapsto (\Gamma(F), \Gamma(\mathsf{m}), \Gamma(\mathsf{e}), \Gamma(\mathsf{b}))
    \end{align*}
    is an equivalence of categories.
\end{proposition}

\begin{proof}
    This is due to the fact that $\Gamma:\Vecbun(M)\to \proj_{C^\infty(M)}$ is an equivalence of categories and that it provides a 1:1-correspondence between the objects, as stated in \ref{objalgcorrprop}. In particular, if $F, F'$ are octonion algebra objects, an isomorphism $f:F\to F'$ of vector bundles is a morphism of octonion algebra objects if and only if it satisfies the pushforward axiom of \ref{octalgobjmorphrels}, which occurs if and only if $\Gamma(F)$ satisfies the pushforward axiom of \ref{octalgmorphrels} which is equivalent to $\Gamma(F)\to\Gamma(F')$ being a morphism of octonion algebras.
\end{proof}

\begin{remark}
    The equivalence is given even more directly by defining octonion algebra objects over an arbitrary category $\CC$, as is done with for instance algebra objects and Hopf algebras. We then need to require the category to be tensor category. This is needed to, among other things, say something about the non-degeneracy of the bilinear form. With morphisms of octonion algebra objects defined as isomorphisms of underlying objects that push forward the multiplication, identity and bilinear form (as in \ref{octalgmorphrels} or \ref{octalgobjmorphrels}) it is immediate that categories of octonion algebra objects are preserved under equivalence of respective underlying categories.

    The fact that the octonion algebra objects of $\proj_R$ are the octonion algebras over $R$ is the content of the last remark in the proof of proposition \ref{objalgcorrprop}.
\end{remark}

\begin{remark}\label{elemoctalgrmk}
    It is natural to ask if the additional structure that an octonion algebra object in $\Vecbun(M)$ has an influence on the underlying bundle. For instance, the admission of an identity $\mathsf{e}:\Lambda^0\to F$ for a rank 8 bundle $F$ gives us a distinguished subbundle of $F$ that is isomorphic to $\Lambda^0$. By this fact, there is a splitting $F\cong \Lambda^0\oplus F'$ with $F'$ a rank 7 bundle. To further understand the global structure, and the properties of the associated octonion algebra, we first inspect the structure of the fibres, i.e., over any given point $p\in M$.
\end{remark}

\begin{lemma}\label{octalgobjfibre}
    Let $M$ be a smooth manifold and let $(F, \mathsf{m}, \mathsf{e}, \mathsf{b})$ be an octonion algebra object of $\Vecbun(M)$. Then, for every point $p\in M$, the data at the fibre, i.e., the quadruple $(F_p, \mathsf{m}_p, \mathsf{e}_p, \mathsf{b}_p)$, define an octonion algebra over $\BR$ with underlying vector space (or module) $F_p$.
\end{lemma}

\begin{proof}
    By definition \ref{octalgobjdef} the restriction $F_p$ is a vector space of dimension 8 over $\BR$ and the maps
    \begin{align*}
        \mathsf{m}_p:F_p\otimes F_p\to F_p, \hspace{6mm} \mathsf{b}_p:F_p\otimes F_p\to \Lambda^0_p, \hspace{6mm} \mathsf{e}_p:\Lambda^0_p\to F_p,
    \end{align*}
    where $\mathsf{b}_p:F_p\otimes F_p\to \Lambda^0_p$ is non-degenerate, make the following diagrams commute:
    \begin{center}
        \begin{tikzcd}
        \Lambda^0_p\otimes F_p \arrow[r, "\mathsf{e}_p\otimes\mathrm{Id}"] \arrow[rd, "\cong"'] & F_p\otimes F_p \arrow[d, "\mathsf{m}_p"] & F_p\otimes \Lambda^0_p \arrow[l, "\mathrm{Id}\otimes \mathsf{e}_p"'] \arrow[ld, "\cong"] \\
                                                                                   & F_p                         &                                                                            
        \end{tikzcd}
        \hspace{8mm}
        \begin{tikzcd}
        F_p\otimes F_p \arrow[r, "\mathsf{b}_p"] \arrow[d, "\mathsf{T}"'] & \Lambda^0_p \\
        F_p\otimes F_p \arrow[ru, "\mathsf{b}_p"']               &     
        \end{tikzcd}
        \end{center}
        \begin{center}
        \begin{tikzcd}
        F_p\otimes F_p\otimes F_p\otimes F_p \arrow[d, "\mathsf{b}_p\otimes \mathsf{b}_p"'] \arrow[r, "{(\lambda_{(\mathsf{m}, \mathsf{b})})_p}"] & \Lambda^0_p \\
        \Lambda^0_p\otimes\Lambda^0_p \arrow[ru, "\cong"']                                                                            &          
        \end{tikzcd}
    \end{center}
    Since $\Lambda^0$ is the trivial bundle we have $\Lambda^0_p=\BR$ which gives us that the above defines a real octonion algebra. Alternatively, passing from $(F, \mathsf{m}, \mathsf{e}, \mathsf{b})$ to its associated octonion algebra via global sections, via proposition \ref{objalgcorrprop}, it suffices to evaluate the given sections at the point $p\in M$.
\end{proof}

\begin{remark}\label{questionremark}
     This local structure provides insight into the restrictions of the global structure, and will be the subject of lemma \ref{relationslemma}. Before this, we can again ask the question of whether an octonion algebra object $F$ has further restrictions on its underlying bundle. With respect to the splitting $F=\Lambda^0\oplus F'$ given by the identity there exist different options for the remaining rank 7 part. We present two examples, one of which will be our main focus later.
\end{remark}

\begin{example}
    Let $M$ be any smooth manifold. Then any octonion algebra $(C, m, e, b)$ over $\BR$ can be extended to an octonion algebra object $((\Lambda^0)^{\oplus8}, \mathsf{m}, \mathsf{e}, \mathsf{b})$ where, for any $p\in M$,
    \begin{align*}
        (\Lambda^0)^{\oplus8}_p=C, \hspace{6mm} \mathsf{m}_p:=m, \hspace{6mm} \mathsf{e}_p:=e, \hspace{6mm} \mathsf{b}_p:=b,
    \end{align*}
    given an isomorphism $\Lambda^0_p\to\BR$. In other words, this extension to the trivial rank 8 bundle is given by assigning to every fibre the same real octonion algebra, namely $(C, b, m, e)$.
\end{example}

\begin{example}
    Let $M$ be a smooth 7-dimensional manifold. We will see that an octonion algebra object with underlying bundle equal to $\Lambda^0\oplus TM$ exists under certain conditions (theorems \ref{g2obstr} and \ref{maintheorem}). Such octonion algebra objects will be the main interest in the study of octonion bundles. Note that this applies to the previous example whenever $M$ is parallelisable, i.e., whenever $TM\cong(\Lambda^0)^{\oplus 7}$.
\end{example}

\begin{remark}
    We continue the local treatment of octonion algebra objects by establishing relations important for next subsection and section \ref{twistsection}.
\end{remark}

\begin{proposition}[Relations]\label{relationslemma}
    Let $(P, m, e, b)$ be an octonion algebra over $C^\infty(M)$. Then any elements $x, y, z\in P$ and $x_i, y_i\in P$, $i=1, 2$, satisfy
    \begin{align}\label{masterformula}
        xy+yx-b(y\otimes 1)x-b(x\otimes 1)y+b(x\otimes y)=0,
    \end{align}
    \begin{align}\label{3formrels}
        \begin{cases}
            b(xy\otimes z)=b(y\otimes \Bar{x}z)\\
            b(xy\otimes z)=b(x\otimes z\Bar{y})\\
            b(xy\otimes \Bar{z})=b(yz\otimes \Bar{x}),
        \end{cases}
    \end{align}
    and
    \begin{align}\label{barmultrel}
        b(x, x)=2x\Bar{x}
    \end{align}
    where the involution is defined as
    \begin{align*}
        \Bar{x}:=b(x\otimes 1)-x.
    \end{align*}
\end{proposition}

\begin{proof}
    First, let $(F, \mathsf{m}, \mathsf{e}, \mathsf{b})$ denote the octonion algebra object of $\Vecbun(M)$ associated to $(P, m, e, b)$. Let $p\in M$ be any point and let $(F_p, \mathsf{m}_p, \mathsf{e}_p, \mathsf{b}_p)$ denote the real octonion algebra given by the octonion algebra object $F$ at $p$, utilising lemma \ref{octalgobjfibre}. On this real octonion algebra, the relations of \cite[Remark 1.2.2]{SV}, \cite[Proposition 1.2.3]{SV} and \cite[Lemma 1.3.2]{SV} hold. In effect, the relations are satisfied on the fibre $F_p$ for any point $p\in M$, hence globally;

    In particular, any identity that holds for elements in a real octonion algebra, expressed using elements of the octonion algebra and the multiplication and bilinear form, hold for any sections of an octonion algebra object. This is since sections are globally defined and evaluate over any point $p\in M$ to elements of a real octonion algebra. This concludes the proof.
\end{proof}

\begin{remark}
    Note that, as is stated in the last paragraph of the proof, the above procedure may be applied to any identities satisfied by a real octonion algebra to get analogous results for the sections of an octonion algebra object over $\Vecbun(M)$.

    The above relations also hold in any octonion algebra over any ring; The first identity, for instance, is given by linearising relation (4-3) of \cite{LPR}. The remaining identities \ref{3formrels} and \ref{barmultrel} almost immediately by definition of the involution.
\end{remark}

\subsection{Octonion bundles: Intrinsic definition and first consequences}\label{octbunintrinsic}

Octonion bundles were first defined in \cite{Gri}. This definition uses a predetermined $G_2$-structure $\varphi$ and, from that, constructs a metric and multiplication on the bundle $\Lambda^0\oplus TM$. We aim to do this without reference to a predetermined $G_2$-structure. We instead define octonion bundles by realising them as octonion algebra objects in the category $\Vecbun(M)$. In particular, we define octonion bundles as a certain class of octonion algebra objects. To do this we first present the following definition. Throughout this section we assume that $M$ is connected.

\begin{definition}[Globally split and globally non-split octonion algebra objects]\label{splitoctalgobj}
    Let $M$ be a smooth manifold and let $(F, \mathsf{m}, \mathsf{e}, \mathsf{b})$ be an octonion algebra object in $\Vecbun(M)$. We say that this object is \emph{globally split} if, for every point $p\in M$, the real octonion algebra $(F_p, \mathsf{m}_p, \mathsf{e}_p, \mathsf{b}_p)$ is split. In the same way we say that the octonion algebra object is \emph{globally non-split} if, for every point $p\in M$, the real octonion algebra $(F_p, \mathsf{m}_p, \mathsf{e}_p, \mathsf{b}_p)$ is non-split.
\end{definition}

To show that there exist globally split respectively non-split octonion algebra objects, and to see how these behave, we first need the following lemma and proposition.

\begin{lemma}\label{sylvesterlemma}
    Let $V$ be a finite-dimensional $\BR$-vector space. Suppose that $\{B_\lambda : \gamma\in(\alpha, \beta)\}$ is a family of non-degenerate symmetric bilinear forms on $V$, varying continuously along $(a, b)\subset\BR$. Then all $B_\lambda$ have the same signature.
\end{lemma}

\begin{proof}
    This is a corollary of Sylvester's law of inertia.
\end{proof}

\begin{proposition}\label{fibresplit}
    Let $(F, \mathsf{m}, \mathsf{e}, \mathsf{b})$ be an octonion algebra object in $\Vecbun(M)$. If $(F_p, \mathsf{m}_p \mathsf{e}_p, \mathsf{b}_p)$ is a non-split (respectively split) real octonion algebra for some point $p\in M$ then $(F, \mathsf{m}, \mathsf{e}, \mathsf{b})$ is a globally non-split (respectively globally split) octonion algebra object in $\Vecbun(M)$, in the sense of definition \ref{splitoctalgobj}.
\end{proposition}

\begin{proof}
    Assume that $(F, \mathsf{m}, \mathsf{e}, \mathsf{b})$ is an octonion algebra object in $\Vecbun(M)$ and let $p\in M$ be a point such that $(F_p, \mathsf{m}_p, \mathsf{e}_p, \mathsf{b}_p)$ is a real non-split (respectively split) octonion algebra. Let $q\in M$ be any other point and denote by $t:[0, 1]\to M$ a continuous path from $t(0)=p$ to $t(1)=q$ (recall that $M$ is connected). This lifts to a family $\{\mathsf{b}_{t(\lambda)}: \; \lambda\in[0, 1]\}$ that varies continuously along $[0, 1]$. By lemma \ref{sylvesterlemma}, the signature remains invariant along this path. Due to this, $\mathsf{b}_q$ is the bilinear form of a real octonion algebra and has the same signature as $\mathsf{b}_p$, so it must be positive-definite (respectively indefinite), i.e., the bilinear form of a real non-split (respectively split) octonion algebra.
\end{proof}

\begin{corollary}\label{splitness}
    Every octonion algebra object over $M$ is either globally non-split or globally split.
\end{corollary}

\begin{remark}
    Due to the equivalence \ref{octalgcatobjequivalence} we can say that an octonion algebra $C\in\catname{octalg}_{C^\infty(M)}$ is globally non-split (respectively split) if the corresponding octonion algebra object $\Gamma^{-1}(C)$ is globally non-split (respectively split). Furthermore, since it is a condition dependent on fibres we can define an equivalent notion of being globally split (respectively globally non-split) for octonion algebras. To do this, we return to the algebraic setting over an arbitrary ring $R$.
\end{remark}

\begin{definition}[Local splitness of octonion algebra]
    Let $(C, m, e, b)$ be an octonion algebra over a ring $R$. Let $\mathfrak{p}$ be a prime ideal of $R$ containing $\frac{1}{2}$. We say that $C$ is \emph{split} (respectively \emph{non-split}) \emph{at $\mathfrak{p}$} if the octonion algebra
    \begin{align*}
        (C\otimes_R\kappa(\mathfrak{p}),\\
        m\otimes\Id_{\kappa(\mathfrak{p})}),\\
        e\otimes\Id_{\kappa(\mathfrak{p})}),\\
        b\otimes\Id_{\kappa(\mathfrak{p})})
    \end{align*}
    over the residue field $\kappa(\mathfrak{p})=R_\mathfrak{p}/\mathfrak{p}R_\mathfrak{p}$ is split (respectively non-split).
\end{definition}

\begin{remark}
    The fact that this defines an octonion algebra over $\kappa(\mathfrak{p})$ follows from the fact that octonion algebras are preserved under change of ring. We can then define $(C, m, e, b)$ as \emph{globally split} (respectively \emph{globally non-split}) if it is split (respectively non-split) at every maximal ideal $\mathfrak{m}$ of $R$.
    
    Due to propositions \ref{evalCprop} and \ref{Fibrelocalprop} this means that we can characterise proposition \ref{splitoctalgobj} algebraically:
\end{remark}

\begin{corollary}
    Any octonion algebra $(C, m, e, b)$ over the ring $C^\infty(M)$ is either globally split or globally non-split.
\end{corollary}

\begin{remark}
    We now present the definition of an octonion bundle. It is an octonion algebra object over a 7-dimensional Riemannian manifold, with bundle structure determined by the tangent bundle $TM$ and with bilinear form relating to the metric $g$ on $M$. In this sense, it is an octonion algebra object that is compatible with the structure of $(M, g)$ as a Riemannian manifold.
\end{remark}

\begin{definition}[Octonion bundle]\label{octbundef}
    Let $M$ be a 7-dimensional manifold with metric $g$. An \emph{octonion bundle} over $(M, g)$ is a globally non-split octonion algebra object $(F, \mathsf{m}, \mathsf{e}, \mathsf{b})$ in $\Vecbun(M)$ such that
    \begin{align*}
        \begin{cases}
            F=\Lambda^0\oplus TM\\
            \mathsf{b}=\mathsf{b}_g\\
            \mathsf{e}:\Lambda^0\hookrightarrow\Lambda^0\oplus TM.
        \end{cases}
    \end{align*}
    where $\mathsf{e}$ denotes the canonical inclusion $\mathsf{e}_p:v_p\mapsto (v_p, 0_p)\in\Lambda^0_p\oplus T_pM$ and where $\mathsf{b}_g$ denotes the bilinear form defined by
    \begin{align*}
        (\mathsf{b}_g)_p(a_p+\alpha_p, d_p+\delta_p)=2a_pd_p+2g_p(\alpha_p, \delta_p)\hspace{5mm} \forall \; a_p+\alpha_p, d_p+\delta_p\in\Lambda^0_p\oplus T_pM.
    \end{align*}
\end{definition}

\begin{remark}
    Since three of the four objects defining an octonion algebra object are fixed in an octonion bundle, the moduli of octonion bundles are simply the moduli of octonion algebra object multiplications $\mathsf{m}$ on $(\Lambda^0\oplus TM, \mathsf{b}, \mathsf{e})$. In turn, we can say that these correspond to the octonion multiplications $m$ on the quadratic module $(\Gamma(\Lambda^0\oplus TM), \Gamma(\mathsf{b}_g))$ with unit given by the unit section $e=1\in C^\infty(M)$ and with bilinear form $b_g=\Gamma(\mathsf{b}_g)$ defined by
    \begin{align*}
        b_g(a+\alpha, d+\delta)=2ad+2g(\alpha, \delta).
    \end{align*}
    A fact used later is that under $\mathsf{b}_g$, or equivalently $b_g$, any section of $\Lambda^0$ is orthogonal to any section of $TM$ in $\Lambda^0\oplus TM$.
\end{remark}

\begin{remark}\label{catremark}
    By an octonion bundle over $M$, for a 7-dimensional manifold $M$, we mean an octonion bundle over $(M, g)$ for some metric $g$ on $M$.

    We will denote the category of octonion bundles over $M$ by $\catname{octbun}(M)$. Note that it is a full subcategory of $\catname{octobj}_{\Vecbun(M)}$. Using the equivalence of categories of proposition \ref{octalgcatobjequivalence} we will mostly view $\catname{octbun}(M)$ as a full subcategory of $\Octalg_{C^\infty(M)}$. In light of this, we introduce the notation $\catname{octbun}_{C^\infty(M)}$ for this full subcategory, i.e., we have the following corollary of \ref{octalgcatobjequivalence}:
\end{remark}

\begin{corollary}\label{octbunalgsequiv}
    There is an isomorphism of categories
    \begin{align*}
        \Gamma:\catname{octbun}(M)\overset{\cong}{\longrightarrow}\catname{octbun}_{C^\infty(M)}.
    \end{align*}
\end{corollary}

\subsection{Octonion bundles constructed from $G_2$-structures}\label{octbunfromg2}

We outline the treatment of $G_2$-structures in terms of octonion bundles conducted in \cite{Gri}. Throughout, $M$ will denote a 7-dimensional Riemannian manifold with metric $g$ and a fixed $G_2$-structure $\varphi$ with $g$ as its associated metric.

\begin{construction}\label{g2octbunconstr}
    In \cite{Gri} Grigorian defines structures there called \emph{octonion bundles} and shows that they correspond to $G_2$-structures. Such a structure is defined as the bundle $\BO M=\Lambda^0\oplus TM$ over $M$, with $\Lambda^0=M\times\BR$ the trivial bundle, endowed with a multiplication defined on its module of sections. In particular, Grigorian defines a \emph{vector cross product with respect to} $\varphi$ on $M$, with $\varphi$ a $G_2$-structure, denoted by $\times_\varphi$ and implicitly defined by
    \begin{align}\label{g2formula}
        g(\alpha\times_\varphi\beta, \gamma)=\varphi(\alpha, \beta, \gamma), \hspace{6mm} \alpha, \beta, \gamma\in \Gamma(TM),
    \end{align}
    see \cite[(3.5)]{Gri}. This implicit definition yields the following explicit description, after taking the relevant dual:
    \begin{align*}
        \alpha\times_\varphi\beta&:=((\alpha\wedge\beta)\lrcorner\varphi)^\flat\\
        &=\varphi(\alpha, \beta, -)^\flat.
    \end{align*}
    With the underlying bundle structure of $\BO M$ one can define a product denoted $\circ_\varphi$ on its sections, i.e., a multiplication on the projective rank 8 module $\Gamma(\BO M)$;
    \begin{align}
        A\circ_\varphi B=ab-g(\alpha, \beta)+a\beta+b\alpha+\alpha\times_\varphi\beta.
    \end{align}
    Here, $A, B\in\Gamma(\BO M)$ are sections of $\BO M$ with $A=a+\alpha$, $B=b+\beta$ such that $a, b\in\Gamma(\Lambda^0)$ and $\alpha, \beta\in\Gamma(TM)$. This is referred to as an \emph{octonion product}. In its original expression, in \cite[Definition 3.4.]{Gri}, Grigorian makes visible the real- and imaginary part of the product, i.e., which part is sent to $\Gamma(\Lambda^0)$ and to $\Gamma(TM)$ respectively:
    \begin{align*}
        ab-g(\alpha, \beta)\in\Gamma(\Lambda^0), \hspace{6mm} a\beta+b\alpha+\alpha\times_\varphi\beta\in\Gamma(TM).
    \end{align*}
\end{construction}

\begin{remark}\label{g2reconstruction}\label{bilinearformdiffrmk}
    The definitions and terminology of section \ref{octbunintrinsic} are of course made to coincide with the vocabulary of Grigorian. This is done mainly in order to provide a categorical and even more algebraic perspective on these octonion bundles and their properties. The definitions rely mostly on algebraic definitions translated to the context of vector bundles, and the further restrictions which yield octonion bundles are not made with reference to an existing $G_2$-structure (hence the name of section \ref{octbunintrinsic}), but rather to compatibility with the Riemannian structure of $M$, or $C^\infty(M)$.
    
    The above construction gives us a bilinear and quadratic form on $\Gamma(\BO M)$ uniquely determined by
    \begin{align*}
        b_g(a+\alpha, a+\alpha)=2q_g(a+\alpha)=2a^2+2g(\alpha, \alpha), \hspace{6mm} a\in\Gamma(\Lambda^0),\; \alpha\in\Gamma(TM).
    \end{align*}
\end{remark}

\begin{remark}\label{crossprodmultrmk}
    We now present some facts concerning construction \ref{g2octbunconstr}. Let $\BO M$ be an octonion bundle defined by $\varphi$ and let $\alpha, \beta\in\Gamma(TM)$ be vector fields. These can be viewed as sections of the octonion bundle $\BO M$ with real part equal to zero. With the definition of the octonion multiplication, we get that
    \begin{align*}
        \alpha\circ_\varphi\beta=-g(\alpha, \beta)+\alpha\times_\varphi\beta.
    \end{align*}
    This gives us that $\alpha\circ_\varphi\beta+g(\alpha, \beta)=\alpha\times_\varphi\beta$, so we get the $G_2$-structure $\varphi$ by equation \ref{g2formula}, i.e.,
    \begin{align*}
        \varphi(\alpha, \beta, \gamma)&=g(\alpha\times_\varphi\beta, \gamma)\\
        &=g(\alpha\circ_\varphi\beta+g(\alpha, \beta), \gamma)\\
        &=g(\alpha\circ_\varphi\beta, \gamma)+g(g(\alpha, \beta), \gamma)\\
        &=g(\alpha\circ_\varphi\beta, \gamma)
    \end{align*}
    where $g(g(\alpha, \beta), \gamma)=0$ since $g(\alpha, \beta)\in\Gamma(\Lambda^0)$ is a section of the trivial bundle and since $\gamma\in\Gamma(TM)$ is taken to be a vector field. This means that we recover the $G_2$-structure uniquely from its associated octonion bundle, so the map sending $G_2$-structures to algebras, or bundles with additional data, is injective.
\end{remark}

\subsection{Correspondence}\label{mainthmsection}

We begin by stating the main motivating result for studying octonion bundles:

\begin{theorem}\label{maintheorem}
    Let $M$ be a Riemannian manifold with vanishing first and second Stiefel-Whitney classes. Fix a $G_2$-metric $g$. Then there is a 1:1-correspondence
    \begin{align*}
        \{\text{$G_2$-structures over $M$ with metric $g$}\}
    \end{align*}
    \begin{align*}
        \updownarrow 1:1
    \end{align*}
    \begin{align*}
        \{\text{Octonion bundles over $(M, g)$}\}
    \end{align*}
\end{theorem}

\begin{remark}\label{pfrmk}
    Both the statement and the proof are in the spirit of Grigorian and are formulated to fit the categorical framework. One direction has already been shown, in construction \ref{g2octbunconstr} and remark \ref{g2reconstruction}. It gives us a map from $G_2$-structures to octonion bundles and that this map is injective. For clarity we show the other direction.

    For the other direction, let $F$ be an octonion bundle with underlying bundle $F=\Lambda^0\oplus TM$ and bilinear form $b=b_g$, i.e., for $a+\alpha=A\in \Gamma(\Lambda^0)\oplus\Gamma(TM)$ we have $b_g(A, A)=2a^2+2g(\alpha, \alpha)$. We then define a 3-form on $TM$ by
    \begin{align*}
        \varphi_F(\alpha, \beta, \gamma):=\frac{1}{2}b(m(\alpha\otimes\beta)\otimes \gamma).
    \end{align*}
    Note the compatibility with remark \ref{crossprodmultrmk} via remark \ref{bilinearformdiffrmk}, as the definition is stated with respect to the bilinear form $b$. We prove that this is a 3-form which on every fibre is equal to the standard form $\varphi_0$ with adequate choice of basis. This will conclude the proof of theorem \ref{maintheorem}.
\end{remark}

\begin{proposition}\label{varphi_Cconstr}
    Let $(\Gamma(\Lambda^0\oplus TM), m, e, b_g)$ be an octonion bundle over $M$. Then
    \begin{align*}
        \frac{1}{2}b_g(m(-\otimes -)\otimes -):\Gamma(TM)\otimes\Gamma(TM)\otimes\Gamma(TM)\to \Lambda^0
    \end{align*}
    determines a positive 3-form, i.e., a $G_2$-structure on $M$.
\end{proposition}

\begin{proof}
    We inspect the involution defined in the statement of lemma \ref{relationslemma}:
    \begin{align*}
        \Bar{A}=b(A\otimes e)-A
    \end{align*}
    where $e\in\Gamma(\Lambda^0)$ is the unit. By the definition of $b$ (and remark \ref{g2reconstruction}) the sections of $\Lambda^0$ are orthogonal to those of $TM$. Since there is a decomposition $A=a+\alpha$, with $a$ and $\alpha$ sections of $\Lambda^0$ and $TM$ respectively, the above expression reads
    \begin{align*}
        \Bar{A}&=b(A\otimes e)-A\\
        &=b((a+\alpha)\otimes e)-A\\
        &=b(a\otimes e)+b(\alpha\otimes e)-A\\
        &=2a-A=a-\alpha.
    \end{align*}
    For $A=\alpha\in\Gamma(TM)\subseteq\Gamma(\Lambda^0\oplus TM)$, a section of $F$ which is zero on $\Lambda^0$, one then gets that $\Bar{\alpha}=-\alpha$. We now apply the relations \ref{3formrels}: Let $\alpha, \beta, \gamma\in\Gamma(TM)$ be sections of $F$ which are zero over $\Lambda^0$. Writing the multiplication $m(\alpha\otimes\beta)$ as $\alpha\beta$ we get that
    \begin{align*}
        b(\alpha\beta\otimes \gamma)&=b(\beta\otimes \Bar{\alpha}\gamma) &&\textnormal{Relation \ref{3formrels}(1)}\\
        &=b(\beta\otimes -\alpha\gamma) &&\textnormal{by the above}\\
        &=-b(\alpha\gamma\otimes \beta) &&\textnormal{$b$ symmetric bilinear form}.
    \end{align*}
    and, using the remaining relations \ref{3formrels}(2) and \ref{3formrels}(3) in the same way as above we get
    \begin{align}\label{propproofrel}
        b(\alpha\beta\otimes \gamma)=-b(\gamma\beta\otimes \alpha), \hspace{6mm} -b(\alpha\beta\otimes \gamma)=-b(\beta\gamma\otimes \alpha).
    \end{align}
    So, using both of the relations of \ref{propproofrel}, one gets
    \begin{align*}
        b(\gamma\beta\otimes \alpha)=-b(\beta\gamma\otimes \alpha)
    \end{align*}
    which constitutes the remaining type of transposition, namely that of the first two entries. Thus we have shown that $b(m(- \otimes -)\otimes -)$, and in effect $\frac{1}{2}b(m(- \otimes -)\otimes -)$, is a 3-form, as all of the constituent operations are $C^\infty(M)$-linear.

    We now show that this 3-form is positive, i.e., that it is isomorphic to the standard form $\varphi_0$ on every fibre. We use that $F$ is an octonion bundle and not only an octonion algebra object; Let $p\in M$ be any point and denote by $F_p$ the real non-split octonion algebra over $\BR$. Denote by $m_p$ and $b_p$ the multiplication and bilinear form associated to the fibre. This octonion algebra is, by remark \ref{uniqueoctalg/field}, isomorphic to the standard octonion algebra $\BR\oplus\BR^7$ with basis $\{e_0, \dots, e_7\}$ given by the multiplication table
    \begin{center}
        \begin{tabular}{|c|c|c|c|c|c|c|c|c|c|c|}
        \hline
         mult. & $e_0$  & $e_1$ & $e_2$ & $e_3$ & $e_4$ & $e_5$ & $e_6$ & $e_7$ \\
         \hline
         $e_0$  & $e_0$ & $e_1$ & $e_2$ & $e_3$ & $e_4$ & $e_5$ & $e_6$ & $e_7$ \\
         \hline
         $e_1$  & $e_1$ & $-e_0$ & $e_3$ & $-e_2$ & $e_5$ & $-e_4$ & $e_7$ & $-e_6$ \\
         \hline
         $e_2$  & $e_2$ & $-e_3$ & $-e_0$ & $e_1$ & $e_6$ & $-e_7$ & $-e_4$ & $e_5$ \\
         \hline
         $e_3$  & $e_3$ & $e_2$ & $-e_1$ & $-e_0$ & $-e_7$ & $-e_6$ & $e_5$ & $e_4$ \\
         \hline
         $e_4$  & $e_4$ & $-e_5$ & $-e_6$ & $e_7$ & $-e_0$ & $e_1$ & $e_2$ & $-e_3$ \\
         \hline
         $e_5$  & $e_5$ & $e_4$ & $e_7$ & $e_6$ & $-e_1$ & $-e_0$ & $-e_3$ & $-e_2$ \\
         \hline
         $e_6$  & $e_6$ & $-e_7$ & $e_4$ & $-e_5$ & $-e_2$ & $e_3$ & $-e_0$ & $e_1$ \\
         \hline
         $e_7$  & $e_7$ & $e_6$ & $-e_5$ & $-e_4$ & $e_3$ & $e_2$ & $-e_1$ & $-e_0$ \\
         \hline
        \end{tabular}
    \end{center}
    and $\BR$-bilinear form defined by 
    \begin{align*}
        b_p(e_i, e_i)=2, \hspace{6mm} b_p(e_i, e_j)=0 \; \text{ for every } \; i\neq j.
    \end{align*}
    On the fibre $F_p$ the resulting 3-form
    \begin{align*}
        \varphi_p(v, w, z):=\frac{1}{2}b_p(m_p(v\otimes w)\otimes z)
    \end{align*}
    is determined by its values on basis vectors, of which the non-zero terms are
    \begin{align*}
        \begin{cases}
            \varphi_p(e_1, e_2, e_3)=1\\
            \varphi_p(e_1, e_4, e_5)=1\\
            \varphi_p(e_1, e_6, e_7)=1\\
            \varphi_p(e_2, e_4, e_6)=1\\
            \varphi_p(e_2, e_5, e_7)=-1\\
            \varphi_p(e_3, e_4, e_6)=-1\\
            \varphi_p(e_3, e_5, e_6)=-1
        \end{cases}
    \end{align*}
    which agrees with the 3-form
    \begin{align*}
        \varphi_0=e^{123}+e^{145}+e^{167}+e^{246}-e^{257}-e^{347}-e^{356}.
    \end{align*}
    This concludes the proof.
\end{proof}

\begin{remark}
    The above proposition gives us the remaining implication in the proof of theorem \ref{maintheorem}. A consequence of the theorem is that any octonion algebra object $F=\Lambda^0\oplus TM$ with bilinear form compatible with the metric $g$ of a Riemannian manifold $M$ is an octonion bundle, which in turn corresponds to an octonion algebra over $C^\infty(M)$. This is due to the fact that the metric $g$ is a positive-definite bilinear form.

    Due to the correspondence it is clear that octonion bundles, with underlying bundle $\Lambda^0\oplus TM$, have a bilinear form given by a $G_2$-metric. Furthermore, thanks to proposition \ref{octalgobjfibre}, the issue of finding a $G_2$-structure translates to the problem of finding an octonion multiplication $m$ that makes $(\Gamma(\Lambda^0\oplus TM), 2g=b_g, m, e)$ into an octonion algebra which, when localised to an arbitrary point, results in a non-split octonion algebra over $\BR$.
\end{remark}

\subsubsection{Categorical perspective}\label{equivalencesection}

We return to the category of $G_2$-structures, as in \ref{g2cat3form} (and, due to \ref{G2catequiv}, equivalently as in \ref{g2catprinc}), that makes the correspondence of theorem \ref{maintheorem} into an equivalence of categories.

\begin{proposition}\label{equivalencetheorem}
    The correspondence of theorem \ref{maintheorem} updates to an isomorphism of categories
    \begin{align*}
        \catname{G2str}(M)\overset{\sim}{\longrightarrow} \catname{octbun}_{C^\infty(M)}.
    \end{align*}
\end{proposition}

\begin{proof}
    We will view the $G_2$-structures as positive 3-forms, since the correspondence is stated in these terms, but note that due to lemma \ref{G2catdefcoh} the result applies to $G_2$-structures in the classical sense.

    For a morphism $f:C\to C'$ of octonion algebras we get that the multiplication, identity and bilinear form on $C'$ are given by $f_*m$, $f_*e$ and $f_*b$ respectively, where $m$, $e$ and $b$ are the multiplication, identity and bilinear form on the octonion algebra $C$. This comes from definition \ref{octalgcatdef}. Assuming that these octonion algebras are octonion bundles we get associated $G_2$-structures $\varphi_C$ and $\varphi_{C'}$ given by the correspondence \ref{maintheorem}. In particular, the morphism of octonion algebras is given by a morphism $f:\Gamma(\Lambda^0\oplus TM)\to \Gamma(\Lambda^0\oplus TM)$ that preserves the (span of the) identity, hence $\Gamma(\Lambda^0)$. As such, the morphism  $f$ descends to an automorphism of $\Gamma(TM)$, which we also denote by $f$. Applying the expression of $\varphi_{C'}$ in terms of $b'$ and $m'$ together with $b'=f_*b$ and $m'=f_*m$ gives us the following:
    \begin{align}\label{pushfequivcalc}
        \varphi_{C'}(\alpha, \beta, \gamma)&\overset{\ref{varphi_Cconstr}}{=}\tfrac{1}{2}b'(m'(\alpha\otimes\beta)\otimes\gamma)&&\nonumber\\
        &=\tfrac{1}{2}(f_*b)((f_*m)(\alpha\otimes\beta)\otimes\gamma)&&\hspace{6mm}\text{definition \ref{octalgpushf}}\nonumber\\
        &=\tfrac{1}{2}b(f^{-1}(f(m(f^{-1}(\alpha)\otimes f^{-1}(\beta)))\otimes f^{-1}(\gamma))    &&\nonumber\\
        &=\tfrac{1}{2}b(m(f^{-1}(\alpha)\otimes f^{-1}(\beta))\otimes f^{-1}(\gamma))&&\nonumber\\
        &=\varphi_C(f^{-1}\otimes f^{-1}\otimes f^{-1})(\alpha, \beta, \gamma)&&\nonumber\\
        &=(f_*\varphi_C)(\alpha, \beta, \gamma).
    \end{align}
    Hence, for any morphism $f:C\to C'$ of octonion bundles we have $\varphi_{C'}=f_*\varphi_C$.

    In the other direction we let $\varphi$ and $\varphi'$ be $G_2$-structures related by $\varphi'=f_*\varphi$ for $f:\Gamma(TM)\to\Gamma(TM)$ an isomorphism, i.e., $f$ a morphism from $\varphi$ to $\varphi'$. Assume that these have respective metrics $g$ and $g'$. Then there is an isomorphism $\id\oplus f:\Gamma(\Lambda^0\oplus TM)\to \Gamma(\Lambda^0\oplus TM)$ of $C^\infty(M)$-modules, which are the underlying modules of the octonion algebras $C_\varphi$ and $C_{\varphi'}$ respectively. The metric is transported accordingly; $g'(\alpha, \beta)=(f_*g)(\alpha, \beta)$, with the pushforward defined like that of the associated bilinear form. Since the morphism preserves the identity $e:\Gamma(\Lambda^0)\to \Gamma(\Lambda^0\oplus TM)$ we inspect the remaining datum, namely the multiplication:
    \begin{align*}
        f_*m_{C_\varphi}(a+\alpha, b+\beta)&=fm_{C_\varphi}(f^{-1}(a+\alpha), f^{-1}(b+\beta))\\
        &=fm_{C_\varphi}(a+f^{-1}(\alpha), b+f^{-1}(\beta))\\
        &=f(ab-g(f^{-1}(\alpha), f^{-1}(\beta))+af^{-1}(\beta)+bf^{-1}(\alpha)+\varphi(f^{-1}(\alpha), f^{-1}(\beta), -)^\flat)\\
        &=ab-g(f^{-1}(\alpha), f^{-1}(\beta))+a\beta+b\alpha+\varphi(f^{-1}(\alpha), f^{-1}(\beta), -)^\flat\\
        &=ab-g'(\alpha, \beta)+a\beta+b\alpha+\varphi'(\alpha, \beta, -)^\flat\\
        &=m_{C_{\varphi'}}(a+\alpha, b+\beta)
    \end{align*}
    With this we have shown that there is a bijection
    \begin{align*}
        \Hom_{\Octalg_{C^\infty(M)}}(C, C')&\longrightarrow\Hom_{\catname{G2Str}(M)}(\varphi, \varphi')\\
        [f\in\Aut(\Gamma(\Lambda^0\oplus TM))]&\longmapsto[f\mid_{\Gamma(TM)}\in\Aut(\Gamma(TM))].
    \end{align*}
    This way, we have shown that there is a fully faithful functor from the category of $G_2$-structures on $M$ to the category of octonion bundles on $M$. Since it is a 1:1-correspondence of objects it constitutes an equivalence of categories and, in particular, an isomorphism of categories.
\end{proof}

\subsubsection{Pfister forms and $G_2$-metrics}\label{Pfisterformsg2metrics}

Due to the correspondence \ref{maintheorem}, a Riemannian manifold $(M, g)$ admits a $G_2$-structure if and only if the vector bundle $\Lambda^0\oplus TM$ with corresponding bilinear form $b_g$ admits the structure of an octonion bundle, i.e., admits an octonion multiplication that is the non-split octonion multiplication on fibres. Due to remark \ref{Pfisterformintro}, this only occurs if the norm associated to $(b_g)_p$ is a Pfister form for every $p\in M$. For an arbitrary Pfister form, the resulting real octonion algebra can be either split or division. Since a Riemannian metric $g$ on $M$ is already guaranteed to be positive-definite, we get that the resulting octonion algebras on fibres $(\Lambda^0_p\oplus TM_p, m_p, e_p, (b_g)_p)$ are division. Thus, we have the following local characterisation of $G_2$-metrics.

\begin{proposition}\label{g2metricpt}
    For $(M, g)$ a Riemannian manifold, the metric $g$ is a $G_2$-metric if and only if for every point $p\in M$ there exists a choice of basis $\{e_1, \dots, e_7\}$ for $T_pM$ such that the quadratic form
    \begin{align*}
        q_p:\Lambda^0_p\oplus T_pM&\to \BR\\
        x&\mapsto g_p(x, x)=\frac{1}{2}(b_g)_p(x, x)
    \end{align*}
    is given by
    \begin{align*}
        q_p=e_0^2+\alpha e_1^2+\beta e_2^2+\alpha\beta e_3^2+\gamma e_4^2+\alpha\gamma e_5^2+\beta\gamma e_6^2+\alpha\beta\gamma e_7^2
    \end{align*}
    for some non-zero $\alpha, \beta, \gamma\in\BR$.
\end{proposition}

\begin{remark}
    A Pfister form
    \begin{align*}
        e_0^2+\alpha e_1^2+\beta e_2^2+\alpha\beta e_3^2+\gamma e_4^2+\alpha\gamma e_5^2+\beta\gamma e_6^2+\alpha\beta\gamma e_7^2
    \end{align*}
    is positive-definite if and only if $\alpha, \beta, \gamma>0$, so in fact the elements $\alpha, \beta$ and $\gamma$ need be positive in the above proposition.
\end{remark}

\section{Geometric twists are algebraic twists}\label{twistsection}

Throughout this section, let $(M, g)$ be a Riemannian manifold admitting a $G_2$-structure $\varphi$. We denote by $C_\varphi$ the octonion algebra over $C^\infty(M)$ given by taking sections of the octonion bundle associated to $\varphi$. We will often omit the multiplication map $m_{C_\varphi}$ and denote the bilinear form by $b_g$.

We use the correspondence of theorem \ref{maintheorem} to show that the twist of a $G_2$-structure, as in theorem \ref{Bryanttwist}, corresponds to the twist of an octonion algebra, as in theorem \ref{twistoctalg}. The idea can be realised by the diagram
\[\begin{tikzcd}
	\varphi & C \\
	{\tilde{\varphi}^A=\varphi_A} & {C^{A, A^{-1}}}
	\arrow[squiggly, tail reversed, from=1-1, to=1-2]
	\arrow["{\text{twist by }A}"', from=1-1, to=2-1]
	\arrow["{\text{twist by }A}", from=1-2, to=2-2]
	\arrow[squiggly, tail reversed, from=2-1, to=2-2]
\end{tikzcd}\]
The result for twists of $G_2$-structures and octonion bundles was originally shown by Grigorian in \cite[Remark 4.9.]{Gri}. We prove it for octonion algebras, mainly using the identities of proposition \ref{relationslemma}.

To state the result we introduce the notion of the algebraic twist of a $G_2$-structure:

\begin{definition}\label{algtwist}
    Let $A\in C_\varphi$ such that $b_g(A, A)=2$. We then call the 3-form $\varphi_A$ defined by
    \begin{align*}
        \varphi_A(\alpha, \beta, \gamma)=\varphi(\alpha A, A^{-1}\beta, \gamma)
    \end{align*}
    for $\alpha, \beta, \gamma\in\Gamma(TM)$ vector fields, the \emph{algebraic twist} of $\varphi$ by $A$. The products on the right hand side are the products as taken in $C_\varphi$.
\end{definition}

\begin{remark}
    Note that if one uses the fact that $\varphi(\alpha, \beta, \gamma)=\frac{1}{2}b_g(\alpha\beta, \gamma)$, then
    \begin{align*}
        \varphi_A(\alpha, \beta, \gamma)=\frac{1}{2}b_g((\alpha A)(A^{-1}\beta), \gamma).
    \end{align*}
    We will see that the algebraic twist of a $G_2$-structure is also a $G_2$-structure. This can be shown directly, using the fact that the resulting is a 3-form and that the statement holds on every fibre, but we will instead utilise theorem \ref{Bryanttwist}. We will refer to the twist $\tilde{\varphi}^A$ used in theorem \ref{Bryanttwist} as the \emph{geometric twist} of the $G_2$-structure $\varphi$ by $A\in C_\varphi$ satisfying $b_g(A, A)=2$. Note that this latter condition is equivalent to the condition stated in theorem \ref{Bryanttwist}, i.e., that $A_r^2+|A_i|^2=1$, where $|A_i|^2=g(A_i, A_i)$.
\end{remark}

\begin{theorem}\label{twistequality}
    Let $A\in C_\varphi$ with $A=A_r+A_i$ for $A_r\in\Gamma(\Lambda^0)$ and $A_i\in\Gamma(TM)$. Assume $b_g(A, A)=2$. Then
    \begin{align*}
        \varphi_A=\tilde{\varphi}^A.
    \end{align*}
\end{theorem}

\begin{proof}
    We prove this by evaluating the twists on a triple $\alpha, \beta, \gamma\in \Gamma(TM)$. First and foremost, note that $b_g(A, A)=2A\Bar{A}$ by the relation \ref{barmultrel}. This means that $A\Bar{A}=1$, so $\Bar{A}=A^{-1}$. Moreover, the product $A\Bar{A}=(A_r+A_i)(A_r-A_i)=A_r^2-A_i^2$. This in turn means that $A_i^2=A_r^2-1$ is real, i.e., $A_i^2\in\Gamma(\Lambda^0)$, a fact that will be used later. It is now possible to decompose the expression for $\varphi_A$:
    \begin{align*}
        \varphi_A(\alpha, \beta, \gamma)&=\frac{1}{2}b_g((\alpha A)(A^{-1}\beta), \gamma)\\
        &=\frac{1}{2}b_g((\alpha A_r+\alpha A_i)(A_r\beta-A_i\beta), \gamma)\\
        &=\frac{1}{2}b_g(A_r^2\alpha\beta-A_r\alpha(A_i\beta)+A_r(\alpha A_i)\beta-(\alpha A_i)(A_i\beta), \gamma)\\
        &=\frac{1}{2}A_r^2b_g(\alpha\beta, \gamma)+\frac{1}{2}A_r(b_g((\alpha A_i)\beta, \gamma)-\frac{1}{2}b_g(\alpha(A_i\beta), \gamma))-\frac{1}{2}b_g((\alpha A_i)(A_i\beta), \gamma)\\
        &=A_r^2\varphi(\alpha, \beta, \gamma)+A_r(\varphi(\alpha A_i, \beta, \gamma)-\varphi(\alpha, A_i\beta, \gamma))-\frac{1}{2}b_g((\alpha A_i)(A_i\beta), \gamma).
    \end{align*}
    The first terms of $\varphi_A(\alpha, \beta, \gamma)$ and $\tilde{\varphi}^A(\alpha, \beta, \gamma)$ are equal. Next, we compare the terms containing exactly one multiple of $A_r$.

    In the case of $\tilde{\varphi}^A(\alpha, \beta, \gamma)$, the $A_r$-term is
    \begin{align*}
        -2A_r(A_i\lrcorner(*\varphi))(\alpha, \beta, \gamma)=-2A_r(*\varphi)(A_i, \alpha, \beta, \gamma).
    \end{align*}
    The result \cite[Lemma 3.6]{Gri} states that
    \begin{align*}
        \alpha(\beta\gamma)-(\alpha\beta)\gamma=2(*\varphi)(-^\sharp, \alpha, \beta, \gamma)
    \end{align*}
    where $-^\sharp$ denotes the dual with respect to the extension of the metric $g$ to $\Gamma(\Lambda^0\oplus TM)$, i.e., $\frac{1}{2}b_g$. Precomposing with the dual with respect to $\frac{1}{2}b_g$, which in this case is denoted by $-^\flat$, we get that
    \begin{align*}
        \frac{1}{2}b_g(\alpha(\beta\gamma)-(\alpha\beta)\gamma, A_i)=2(*\varphi)(A_i, \alpha, \beta, \gamma)
    \end{align*}
    Since $(*\varphi)$ is a 4-form, a property of the Hodge dual, we get that
    \begin{align*}
        -A_r\cdot 2(*\varphi)(A_i, \alpha, \beta, \gamma)&=-2A_r(*\varphi)(\gamma, \alpha, A_i, \beta)\\
        &= -A_r\frac{1}{2}b_g( \alpha(A_i\beta)-(\alpha A_i)\beta, \gamma )=\\
        &=A_r(\varphi(\alpha A_i, \beta, \gamma)-\varphi(\alpha, A_i\beta, \gamma)).
    \end{align*}
    This is precisely the $A_r$-term of $\varphi_A(\alpha, \beta, \gamma)$. It now remains to show that the last terms match, i.e., that
    \begin{align}\label{lasttermeq}
        A_i^2\varphi(\alpha, \beta, \gamma)+2(A_i^\flat\wedge(A_i\lrcorner\varphi)(\alpha, \beta, \gamma))=-\frac{1}{2}b_g((\alpha A_i)(A_i\beta), \gamma).
    \end{align}
    This is done by calculating $2(A_i^\flat\wedge(A_i\lrcorner\varphi)(\alpha_1, \alpha_2, \alpha_3)$, where we let $\alpha_1=\alpha$, $\alpha_2=\beta$ and $\alpha_3=\gamma$;
    \begin{align*}
        2(A_i^\flat\wedge(A_i\lrcorner\varphi)(\alpha_1, \alpha_2, \alpha_3))&=2\frac{1}{1!2!}\sum_{\sigma\in\mathfrak{S}_3}A_i^\flat(\alpha_{\sigma(1)})\cdot\varphi(A_i, \alpha_{\sigma(2)}, \alpha_{\sigma(3)})\hspace{8mm} \text{definition of $A_i^\flat$}\\
        &=\frac{1}{2}(b_g( A_i, \alpha)\varphi(A_i, \beta, \gamma)-b_g(A_i, \alpha)\varphi(A_i, \gamma, \beta)\\
        &+b_g( A_i, \beta)\varphi(A_i, \gamma, \alpha)-b_g(A_i, \beta)\varphi(A_i, \alpha, \gamma)\\
        &+b_g( A_i, \gamma)\varphi(A_i, \alpha, \beta)-b_g(A_i, \gamma)\varphi(A_i, \beta, \alpha))\hspace{8mm} \text{$\varphi$ is a 3-form}\\
        &=b_g(A_i, \alpha)\varphi(A_i, \beta, \gamma)+b_g(A_i, \beta)\varphi(A_i, \gamma, \alpha)+b_g(A_i, \gamma)\varphi(A_i, \alpha, \beta).
    \end{align*}
    We inspect the last term of this expression by applying $\varphi(A_i, \alpha, \beta)=\frac{1}{2}b_g(A_i\alpha, \beta)$ and the formulae of lemma \ref{relationslemma}:
    \begin{align*}
        b_g(A_i, \gamma)\varphi(A_i, \alpha, \beta)&=\frac{1}{2}b_g(A_i, \gamma)b_g(A_i\alpha, \beta)   \hspace{8mm}\ref{octalglemma}\\
        &=\frac{1}{2}(b_g(A_i(A_i\alpha), \gamma\beta )+b_g(A_i\beta, \gamma(A_i\alpha)))   \\
        &=\frac{1}{2}(b_g(A_iA_i\alpha, \gamma\beta ) +b_g(A_i\beta, \gamma(A_i\alpha)))   \hspace{8mm}\text{$A_i^2$ real and \ref{3formrels}}\\
        &=\frac{1}{2}(A_i^2b_g( \alpha\Bar{\beta}, \gamma )+b_g( (A_i\beta)\overline{(\alpha A_i)}, \gamma ))   \hspace{8mm}\text{$\Bar{\alpha_k}=-\alpha_k$, $\overline{\alpha A_i}=\Bar{A_i}\Bar{\alpha}$}\\
        &=-A_i^2\varphi(\alpha, \beta, \gamma)+\frac{1}{2}b_g((A_i\beta)(\alpha A_i), \gamma)   \hspace{8mm}\ref{masterformula}\\
        &=-A_i^2\varphi(\alpha, \beta, \gamma)\\
        &+\frac{1}{2}b_g(-(\alpha A_i)(A_i\beta)+b_g(\alpha A_i, 1)(A_i\beta)\\
        &+b_g(A_i\beta, 1)(\alpha A_i)-b_g(\alpha A_i, A_i\beta), \gamma)   \hspace{8mm}\text{$b_g(1, \gamma)=0$}\\
        &=-A_i^2\varphi(\alpha, \beta, \gamma)-\frac{1}{2}b_g((\alpha A_i)(A_i\beta), \gamma)+\frac{1}{2}b_g( \alpha A_i, 1)b_g(A_i\beta, \gamma)\\
        &+\frac{1}{2}b_g(A_i\beta, 1)b_g(\alpha A_i, \gamma)   \hspace{8mm}\ref{3formrels}\\
        &=-A_i^2\varphi(\alpha, \beta, \gamma)-\frac{1}{2}b_g((\alpha A_i)(A_i\beta), \gamma)-\frac{1}{2}b_g(A_i, \alpha)b_g(A_i\beta, \gamma)\\
        &-\frac{1}{2}b_g(A_i, \beta)b_g(\alpha A_i, \gamma)
    \end{align*}
    The last term equals
    \begin{align*}
        -\frac{1}{2}b_g(A_i, \beta)b_g(\alpha A_i, \gamma)&=-b_g( A_i, \beta)\varphi(\alpha, A_i, \gamma)\\
        &=-b_g(A_i, \beta)\varphi(A_i, \gamma, \alpha).
    \end{align*}
    This then gives us that
    \begin{align}\label{thirdterm}
        b_g( A_i, \gamma)\varphi(A_i, \alpha, \beta)&=-A_i^2\varphi(\alpha, \beta, \gamma)-\frac{1}{2}b_g((\alpha A_i)(A_i\beta), \gamma)\notag\\
        &-b_g( A_i, \alpha)\varphi(A_i, \beta, \gamma)-b_g( A_i, \beta)\varphi(A_i, \gamma, \alpha).
    \end{align}
    Using this in the expression for $2(A_i^\flat\wedge(A_i\lrcorner\varphi))(\alpha, \beta, \gamma)$ we get that two terms are cancelled by those appearing in the above expression of $2g(A_i, \gamma)\varphi(A_i, \alpha, \beta)$, i.e.,
    \begin{align*}
        2(A_i^\flat\wedge(A_i\lrcorner\varphi)(\alpha, \beta, \gamma)&=b_g( A_i, \alpha)\varphi(A_i, \beta, \gamma)+b_g( A_i, \beta)\varphi(A_i, \gamma, \alpha)+b_g( A_i, \gamma)\varphi(A_i, \alpha, \beta)\hspace{6mm}\ref{thirdterm}\\
        &=b_g( A_i, \alpha)\varphi(A_i, \beta, \gamma)+b_g( A_i, \beta)\varphi(A_i, \gamma, \alpha)-A_i^2\varphi(\alpha, \beta, \gamma)\\
        &-\frac{1}{2}b_g((\alpha A_i)(A_i\beta), \gamma)-b_g( A_i, \alpha)\varphi(A_i, \beta, \gamma)-b_g( A_i, \beta)\varphi(A_i, \gamma, \alpha)\\
        &=-A_i^2\varphi(\alpha, \beta, \gamma)-\frac{1}{2}b_g((\alpha A_i)(A_i\beta), \gamma).
    \end{align*}
    This gives us the equality \ref{lasttermeq}. Since it constitutes equality of the last terms of $\tilde{\varphi}^A$ and $\varphi_A$ this concludes the proof.
\end{proof}

\begin{remark}
    This means that since the geometric twist $\tilde{\varphi}^A$ is a $G_2$-structure, the algebraic twist $\varphi_A$ is also a $G_2$-structure, and this holds for any $A\in C_\varphi$ with $b_g(A, A)=2$.

    Note that definition \ref{algtwist} of the algebraic twist of a $G_2$-structure is motivated by the correspondence \ref{maintheorem}:

    Consider the associated octonion algebra $C_\varphi$ with multiplication $\circ_\varphi$. The octonion algebra $C_\varphi^{A, A^{-1}}$ has the same underlying quadratic module as $C_\varphi$ but with multiplication defined by
    \begin{align*}
        B\circ_\varphi^A C=(B\circ_\varphi A)\circ_\varphi (A^{-1}\circ_\varphi C)
    \end{align*}
    and due to theorem \ref{maintheorem} this corresponds to the $G_2$-structure $\varphi_A$. This correspondence and theorem \ref{twistequality} gives us the following characterisation.
\end{remark}

\begin{corollary}
    For $A\in C_\varphi$ with $b_g(A, A)=2$ there is equality
    \begin{align*}
        C_{\tilde{\varphi}^A}=C_{\varphi_A}=C_\varphi^{A, A^{-1}}.
    \end{align*}
\end{corollary}

\section{Generalisation: Octonion Algebras as $G_2$-structures}\label{generalisationsection}

In this section we will find in what way an arbitrary octonion algebra over $C^\infty(M)$ can be viewed as some kind of $G_2$-structure. The result of theorem \ref{maintheorem}, and in particular the equivalence of proposition \ref{equivalencetheorem}, taken together with the question asked in remarks \ref{elemoctalgrmk} and \ref{questionremark}, gives us the following possible generalisation of the notion of a $G_2$-structure:

\begin{definition}
    Let $M$ be a smooth manifold and let $E$ be a (smooth) rank 7 vector bundle over $M$. Then a \emph{$G_2$-structure on $E$} is a $G_2$-reduction of the frame bundle $F(E)$ of $E$.
\end{definition}

We know that these exist in the case $E=TM$ and $E=(\Lambda^0)^{\oplus 8}$. This motivates the following definition.

\begin{definition}[Category of generalised $G_2$-structures]\label{geng2catprinc}
    Let $M$ be a smooth connected manifold. Let the category denoted by $\catname{fG2Str}(M)$ consist of the following data:
    \begin{itemize}
        \item Objects: $G_2$-structures viewed as $G_2$-reductions of the frame bundle $F(E)$ of a rank 7 bundle $E$ on $M$.
        \item Morphisms: Let $P, P'$ be $G_2$-principal bundles with inclusions
        \begin{align*}
            P\hookrightarrow F(E), \hspace{6mm} P'\hookrightarrow F(E').
        \end{align*}
        We say that an isomorphism $\mathrm{f}:E\overset{\cong}{\to}E'$ is a morphism $P\to P'$ if $P'=\mathrm{f}P$.
    \end{itemize}
    We refer to it as the \emph{category of formal $G_2$-structures}.
\end{definition}

\begin{remark}
    We can also define a definite 3-form over any rank 7 bundle $E$ as an element $\varphi_E$ of $\bigwedge^3\Gamma(E^\vee)$ such that $(\varphi_E)_p$ is in the $\GL_7(\BR)$-orbit of $\varphi_0$ after choice of local basis. This also means that there is an analogous notion of category of formal definite 3-forms, with objects
    \begin{align*}
        \bigsqcup_{E\in\Vecbun_7(M)}\bigwedge_+^3\Gamma(E^\vee)
    \end{align*}
    and morphisms given by $C^\infty(M)$-linear isomorphisms $f:\Gamma(E)\overset{\cong}{\to}\Gamma(E')$.

    We show that the result of proposition \ref{G2definite3form} generalises:
\end{remark}

\begin{proposition}
    Let $M$ be a smooth manifold. Then the functor
    \begin{align*}
        \bigsqcup_{E\in\Vecbun_7(M)}\bigwedge_+^3\Gamma(E^\vee)&\longrightarrow \mathrm{ob}(\catname{fG2str}(M))\\
        \varphi_E&\longmapsto F_{\varphi_E}
    \end{align*}
    is an equivalence, where
    \begin{align*}
        (F_{\varphi_E})_p:=\{u\in \Hom(\BR^7, E_p) \; : \; u^*((\varphi_E)_p)=\varphi_0\}
    \end{align*}
\end{proposition}

\begin{proof}
    The proof is analogous to the proof of proposition \ref{G2catequiv}, and uses the generalised statements of lemma \ref{G2catdefcoh} and proposition \ref{G2definite3form} which are also proven analogously.
\end{proof}

\begin{remark}
    This is the result needed to prove that there is an equivalence between the category of formal $G_2$-structures and the category of globally non-split octonion algebra objects:
\end{remark}

\begin{theorem}\label{generalequiv}
    There is an equivalence of categories
    \begin{align*}
        \Octalg_{C^\infty(M)}^+&\longrightarrow\catname{fG2str}(M)\\
        (C, b, m, e)&\longmapsto \varphi_C:=\frac{1}{2}b(m(-\otimes-)\otimes-)\\
        [f:C\overset{\cong}{\to}C']&\longmapsto [f\mid_{C\ominus eC^\infty(M)}:C\ominus eC^\infty(M)\overset{\cong}{\to}C'\ominus e'C^\infty(M)]
    \end{align*}
    where $\Octalg_{C^\infty(M)}^+$ denotes the category of globally non-split octonion algebras over $C^\infty(M)$.
\end{theorem}

\begin{proof}
    Fixing a vector bundle $E$ we begin by showing that there is a 1:1-correspondence between the $G_2$-structures on $E$ and the non-split octonion algebras with underlying module $C^\infty(M)\oplus \Gamma(E)$ (every octonion algebra will have a module of this form due to the existence of identity, see remark \ref{questionremark}). The inverse is, as before, given by
    \begin{align*}
        m_{\varphi_E}(A, D):=ad-g_\varphi(\alpha, \delta)+a\delta+d\alpha+(\varphi_E(\alpha, \delta, -))^\flat.
    \end{align*}
    and the rest consists of verifying that for any isomorphism $f:C\to C'$ of globally non-split octonion algebras there is an equality
    \begin{align*}
        f_*\varphi_C(\alpha, \beta, \gamma)=\frac{1}{2}f_*b_C(f_*m_C(\alpha\otimes\beta)\otimes\gamma).
    \end{align*}
    This is symbolically identical to the calculation in \ref{pushfequivcalc} which concludes the proof.
\end{proof}

\begin{remark}
    With this, it is reasonable to ask whether the obstructions to the existence of a $G_2$-structure on a vector bundle $E$ are analogous to those of $G_2$-structures as defined in \ref{gstruct} (i.e., the $G_2$-structures over the tangent bundle $TM$). In particular, we can state the following:
\end{remark}

\begin{proposition}
    Let $E$ be a vector bundle of rank 7 on $M$. The Stiefel-Whitney classes $w_1(E)$ and $w_2(E)$ of the vector bundle $E$ vanish if and only if there exists a $G_2$-structure on $E$.
\end{proposition}

\begin{proof}
    This is seen to be the case by inspecting the proof of this obstruction, a source of which is given by \cite{gray_vector_1969}. In particular, Gray states in \cite[Theorem (3.1)]{gray_vector_1969} that if a vector cross product of type III on a vector bundle $E\overset{p}{\to}M$ of rank 7 arises then the Stiefel-Whitney classes
    \begin{align*}
        w_1(E)\in H^1(M;\BZ/2\BZ), \; w_2(E)\in H^2(M;\BZ/2\BZ), \; W_3(E)\in H^3(M;\BZ), \; W_5(E)\in H^5(M;\BZ)
    \end{align*}
    vanish. This cross-product corresponds to the octonion multiplication on the bundle $\Lambda^0\oplus E$ by linear extension since, in \cite{gray_vector_1969}, a cross product of type III is defined as the multiplication of a non-split composition algebra of dimension 8 with the real part omitted.

    Together with \cite[Theorem (3.2)]{gray_vector_1969} this means that $F=\Lambda^0\oplus E\overset{\pi\oplus p}{\to}M$ is an octonion algebra object if and only if
    \begin{align*}
        w_1(F)=0 \;\text{ and } \; w_2(F)=0.
    \end{align*}
    In other words, using \ref{generalequiv}, it is indeed the case that the vanishing of $w_1$ and $w_2$ is the obstruction for a vector bundle admitting a $G_2$-structure.
\end{proof}

\begin{remark}
    Further elaborating on possible generalisations of the equivalence between general notions of $G_2$-structures and octonion algebras, corollary \ref{splitness} gives us that there is a decomposition of the category of octonion algebras into full subcategories the following way:
    \begin{align*}
        \Octalg_{C^\infty(M)}=\Octalg_{C^\infty(M)}^+\sqcup\Octalg_{C^\infty(M)}^-.
    \end{align*}
    Here $\Octalg_{C^\infty(M)}^-$ denotes the category of globally split octonion algebras. These subcategories can be studied in their own right, since there exists no morphism from a globally non-split octonion algebra to a globally split octonion algebra (or vice versa) since the signature of a bilinear form is preserved under isomorphism.

    Using this perspective on the other side of the equivalence we see that it is possible to define a split $G_2$-structure on a vector bundle $E$ of rank 7 as a 3-form $\sigma\in\bigwedge^3\Gamma(E)^\vee$ which is locally in the $\GL_7(\BR)$-orbit of
    \begin{align*}
        \sigma_0=e^{123}-e^{145}-e^{167}-e^{246}+e^{257}-e^{347}-e^{356}.
    \end{align*}
    This will be fixed by the split version, sometimes denoted by $G_2^*$, of the real group of type $G_2$, which is non-compact. Denoting the category of split formal $G_2$-structures by $\catname{fG2*str}(M)$ can now state the most general version of the equivalence theorem:
\end{remark}

\begin{theorem}
    There is an equivalence of categories
    \begin{align*}
        \Octalg_{C^\infty(M)}&\overset{\sim}{\longrightarrow}\catname{fG2str}(M)\sqcup\catname{fG2*str}(M)
    \end{align*}
\end{theorem}

\begin{remark}
    This is proven again in the same way. We now have that any octonion algebra over the ring $C^\infty(M)$, for $M$ a smooth manifold, can be realised as a formal either split or non-split $G_2$-structure. Due to this equivalence, in fact only using the correspondence on the level of objects, we get that the classification of split $G_2$-structures is the same as that of non-split $G_2$-structures, by passing to the associated octonion algebra. In particular, we get that if $\varphi_C$ is any split or non-split $G_2$-structure with associated octonion algebra $C_\varphi=C$, any other such $G_2$-structure lying in the same metric class as $\varphi_C$ is realised as the twist $\varphi_{C, A}$ defined by
    \begin{align*}
        \varphi_{C, A}(B, C, D)=\varphi_C(BA, A^{-1}C, D)
    \end{align*}
    by some octonion $A\in C$ with $b_{C_\varphi}(A, A)=2$.
\end{remark}

\begin{example}
    An instance of a formal split $G_2$-structure is given by the trivial bundle $(\Lambda^0)^{\oplus8}$. We can define a 3-form $\sigma$ on the bundle by picking a basis and letting it remain constant, so that we can identify $\sigma_0$ with $\sigma$ everywhere. This is equivalent to giving $(\Lambda^0)^{\oplus8}_p$ the structure of the standard split octonions for every point $p\in M$.
\end{example}

\printbibliography

\end{document}